\documentclass[12pt,reqno]{article}
\usepackage[body={6.5in,9.0in},left=1in,top=1in,centering]{geometry}
\usepackage{graphicx}
\usepackage{amssymb}
\usepackage{epstopdf}
\usepackage{amssymb,amsmath,amsthm,amsfonts,color}
\usepackage[colorlinks=true,urlcolor=blue,linkcolor=blue,citecolor=blue,pdfstartview=FitH]{hyperref}
\usepackage{xcolor}
\usepackage{authblk,pstricks}
\usepackage{tikz}
\usetikzlibrary{shapes,arrows,positioning,fit,backgrounds,calc}
\usepackage{pgfplots}
\usepackage{subcaption}
\pgfplotsset{compat=1.16}
\usepackage[sort,comma]{natbib} 
\usepackage[title]{appendix}
\usepackage{mathrsfs}
\newcmykcolor{orange}{0 .61 0.87 0}
\newcmykcolor{navyblue}{0.94 0.74 0 0}
\newcmykcolor{peach}{0 0.50 0.70 0}
\DeclareMathOperator*{\argmax}{arg\,max} 
 
\newtheorem{thm}{Theorem}[section]

\newtheorem{prop}[thm]{Proposition}
\newtheorem{lem}[thm]{Lemma}
\theoremstyle{definition}
\newtheorem{defn}[thm]{Definition}

\newtheorem{cnd}[thm]{Condition}

\newtheorem{rem}[thm]{Remark}
\newtheorem{example}[thm]{Example}
\newcommand{\thmref}[1]{Theorem~{\rm \ref{#1}}}
\newcommand{\lemref}[1]{Lemma~{\rm \ref{#1}}}

\newcommand{\cndref}[1]{Condition~{\rm \ref{#1}}}
\newcommand{\propref}[1]{Proposition~{\rm \ref{#1}}}
\newcommand{\defref}[1]{Definition~{\rm \ref{#1}}}
\newcommand{\remref}[1]{Remark~{\rm \ref{#1}}}

\makeatletter \@addtoreset{equation}{section}

\allowdisplaybreaks

\def\R{\ensuremath {\mathbb R}}
\newcommand{\EE}{\mathbb E}

\newcommand{\NN}{\mathbb N}
\newcommand{\RR}{\mathbb R}
\def\P{\ensuremath{\mathbb P}}
\newcommand{\F}{\mathcal F}
\newcommand{\I}{{\mathcal I}}
\newcommand{\A}{{\mathcal A}}
\newcommand{\AF}{\mathcal A_{\mathrm F}}
\newcommand{\AI}{\mathcal A_{\mathrm I}}

\newcommand{\cR}{\mathcal R}
\newcommand{\lakj}{\lambda_{k_{j}}}
 
\newcommand{\e}{\varepsilon}
\newcommand{\vphi}{\varphi}

\newcommand{\id}{\mathtt{id}}
 \newcommand{\la}{\lambda} \newcommand{\La}{\mathfrak P}
\newcommand{\wdh}{\widehat}
\newcommand{\wdt}{\widetilde}

\newcommand{\lan}{\langle} \newcommand{\ran}{\rangle}

\renewcommand{\hat}{\widehat}

\newcommand{\E}{{\cal E}}

\newcommand{\comment}[1]{} 

\title{Long-Term Average Impulse Control with Mean Field Interactions\thanks{This research was supported by the Simons Foundation under grant number 8035009.}}
\author[1]{K.L. Helmes}
\author[2]{R.H. Stockbridge}
\author[2]{C. Zhu}
\affil[1]{\small Institute for Operations Research, Humboldt University of Berlin, Spandauer Str. 1, 10178, Berlin, Germany, {\tt helmes@wiwi.hu-berlin.de}}
\affil[2]{Department of Mathematical Sciences,   University of Wisconsin-Milwaukee,   Milwaukee, WI 53201,   USA,   {\tt stockbri@uwm.edu}, {\tt zhu@uwm.edu}}
 
\begin{document}
\maketitle 

\begin{abstract}
This paper analyzes and explicitly solves a class of long-term average impulse control problems with a specific mean-field interaction.  The underlying process is a general one-dimensional diffusion with appropriate boundary behavior.   The model is motivated by applications such as  the optimal long-term management of renewable resources and financial portfolio management.  Each individual agent seeks to maximize her  long-term average reward, which consists of a running reward and income from discrete impulses, where the  unit intervention price depends on the market through a stationary supply rate, the specific mean field variable to be considered. In a competitive market setting, we establish the existence of and explicitly characterize an equilibrium strategy within a large class of policies under mild conditions.  Additionally, we formulate and solve the mean field control problem, in which   agents cooperate with each other, aiming to realize a common maximal long-term average profit.  To illustrate the theoretical results, we examine a stochastic  logistic growth model and a population growth model  in a stochastic environment  with impulse control.
 

\end{abstract}
 
\section{Introduction} \label{sect-intro}
 
This paper considers and explicitly solves a long-term average stochastic impulse control problem with a particular type of mean-field interaction.  Our motivation stems from two sources.  The first is the applications in natural resource management, specifically in the context of optimal and sustainable harvesting strategies.  The second is mathematical in nature.  It concerns 
   (a) the important but subtle interplay between two revenue streams, the incomes from a running reward  and from impulse decisions, 
  and (b) exploring a direct approach, using renewal theory and the renewal reward theorem, to analyze such impulse control problems  with a special mean field interaction that will be described in detail momentarily. This approach differs from the general and well established principle: ``Set up the proper HJB/QVI of the model, couple it with the corresponding Fokker-Planck equation, and apply a fixed point argument.''

Let us now introduce the problem. In the absence of controls, the dynamics of a one-dimensional state process -- which may describe the evolution of some renewable resource  -- is modeled by  a one-dimensional diffusion process  on an interval $\I \subset \RR$  
\begin{equation} \label{e:X0}
d X_{0}(t) = \mu(X_{0}(t)) dt +\sigma(X_{0}(t)) dW(t), \ \ X_{0}(0) = x_{0},
\end{equation}  
where $x_{0}\in \I$ is an arbitrary but fixed point throughout the paper,  $W$ is a one-dimensional standard Brownian motion, and the drift and diffusion are given by the functions $\mu$ and $\sigma$, respectively.  The diffusion process is assumed to have certain boundary behavior; 
see Condition \ref{diff-cnd} for details.

Furthermore, an  individual agent wants to specify when and by how much the state of the process should be reduced to achieve economic benefits.  Her strategy is modeled by an impulse control $R : =\{(\tau_{k}, Y_{k}), k =1,2,\dots\}$ such that for each $k\in \NN$,  $\tau_{k}$ is the time of the $k$th intervention and $Y_{k}$ is the size of the intervention.  The resulting  controlled process $X^{R}$ satisfies  
\begin{equation} \label{e:X} 
\begin{aligned}
X^{R}(t) & =x_{0} + \int_{0}^{t} \mu(X^{R}(s))ds + \int_{0}^{t} \sigma(X^{R}(s))dW(s) 
 - \sum_{k=1}^{\infty}I_{\{\tau_{k} \le t\}} Y_{k},  \quad t\ge 0.
\end{aligned}\end{equation} 
A fundamental quantity associated with each policy $ R$ is the long-term average supply rate of product to the market given by 
\begin{align}
\label{e:kappaQ}
 \kappa^{R} :&  =   \limsup_{t\to\infty}   \frac1t  \EE\Bigg[  \sum_{k=1}^{\infty}I_{\{\tau_{k} \le t\}}  Y_k\Bigg]  =  \limsup_{t\to\infty}   \frac1t \EE\Bigg[  \sum_{k=1}^{\infty}I_{\{\tau_{k} \le t\}} (X^{R}(\tau_{k}-) - X^{R}(\tau_{k}))\Bigg]. 
\end{align}

Regarding the market structure, we assume that the market's supply side comprises a continuum of agents, each with the same state dynamics (in the absence of control) and reward structure as the individual agent under consideration.
An individual agent's reward depends not only on her own impulse strategy $R$, but also  crucially  on the market's long-term average supply rate $\kappa^Q$, which results when all other agents adopt policy $Q$.  The supply rate $\kappa^Q$ is the key mean-field interaction that determines the market price of the product of interest through a continuous function $\vphi$. 

Given a positive fixed cost $K $ for each intervention,  a running reward  function $c$, a product supply rate to the market $z : = \kappa^Q$, and a price function $\varphi$, the  reward functional for  an individual agent who adopts policy $R$ is her expected long-term average profit:
\begin{equation}\label{e:reward-fn} \begin{aligned}
J(R; z) : &= \liminf_{t\to\infty}  t^{-1} \EE \bigg[\int_{0}^{t} c(X^{R}(s)) ds \\&  \hspace{1in}+  \sum_{k=1}^{\infty} I_{\{\tau_{k} \le t\}} (\varphi( z) (X^{R}(\tau_{k}-)-X^{R}(\tau_{k})) - K)\bigg].
\end{aligned}
\end{equation}   
The presence of the positive fixed cost $K$ for each intervention implies that $\lim_{k \to \infty} \tau_k = \infty$ a.s.; otherwise $J(R;z) = -\infty$, so $R$ will not be optimal.

 The expected long-term average revenue has two components consisting of a running reward and the reward obtained from the impulse control.
The function $c$ quantifies the running rewards based on the values of the controlled process. In  the context of  harvesting problems, the function $c$ can represent the utility derived from maintaining desirable state values $X^{R}(s)$  at time $s$, as well as the state's contribution to the overall ecosystem's stability.
 For example, the function $c$ can be used to model a subsidy or a stream of carbon credits for managing large tracts of forest. 

The reward from each control action is the net price $\vphi(z)$ times the difference
in states at the time of intervention minus the fixed cost $K$. 
In general, the impulse cost for production-type problems has a fixed component and a variable cost. For the models considered in this paper, the variable cost is
proportional to the size of the intervention and the net price  $\vphi(z)$ subsumes this proportional cost.
  The fixed cost for an intervention  in \eqref{e:reward-fn} makes the problem one of impulse control. 
 The optimal policy thus involves discrete interventions rather than continuous adjustments, ensuring effective product management while maximizing  the overall profit rate.

A fundamental assumption on the model is that each price $\vphi(z)$ in the range of supply rates $z$ is large enough so that some active policy yields a better long-term average reward than the ``do-nothing'' policy that never intervenes.  Such prices are called {\em feasible.}\/  The collection of feasible prices will be denoted by $ \La$, for which a functional representation will be given in \eqref{e:set-Lambda-defn}.

Due to the presence of the nonnegative running reward rate $c$, the interplay between $c$ and the production rate $\mu$ is one of the essential and important   features of
the model and requires careful analysis.  Although the case of a negative function $c$ is also relevant in applications such as inventory control and industrial animal husbandry, a negative running or holding cost term in fact simplifies both the analysis and the characterization of optimal controls near the right boundary of the state space. Specifically, a negative $c$ prizes interventions that keep the controlled process away from the right boundary, thereby avoiding challenges  associated with boundary behavior of the underlying diffusion. By contrast, a nonnegative running reward function $c$ may encourage the process to approach the right boundary, necessitating a more careful examination of several elementary results. 

 Two problems will be investigated  in this paper. First, we consider a competitive market where the agents compete with each other. Our goal is to establish the existence of an equilibrium strategy under the long-term average criterion in a large class of admissible policies. In other words, we wish to determine whether there exists an  admissible strategy $Q $ so that for all admissible $R$ in the class of policies
\begin{equation}\label{e:sec2game}
J(R; \kappa^Q) \le J(Q; \kappa^Q).
\end{equation} 
This is a mean field game (MFG) problem with impulse control.  It implies that, given the stationary supply rate   $\kappa^{Q} $ of an  equilibrium policy $Q$, an   individual agent has no incentive to deviate from the policy $Q$.  

The second problem addressed in this paper is a mean field control (MFC) problem, in which the agents in the market cooperate with each other, aiming to achieve a common maximum long-term average reward.  In other words, the goal is to find an admissible strategy $Q^{\star} $
 so that for all admissible $R$  
\begin{equation}\label{e:sec2control}
 J(R; \kappa^{R}) \le J(Q^{\star}; \kappa^{Q^{\star}}).
\end{equation} 
 Note that the formulation \eqref{e:sec2control} captures the fact that cooperation among all agents in the market results in a single stationary supply rate.   Hence we can regard \eqref{e:sec2control} as a {\em central planner optimization problem.} Mathematically, the reward functional $J(R; \kappa^{R})$ of \eqref{e:reward-fn} depends on the long-term average supply rate of the policy $R$ which in turn depends on the distribution of the controlled process $X^R$. In light of the discussion in Chapter 6 of \cite{CarmoD-18I},   \eqref{e:sec2control} is therefore a mean field control   problem. 
 
\thmref{thm-main-MFGC} provides a set of conditions under which both MFG and MFC are solvable within the broad class of admissible policies.  Importantly, it is the same set of sufficient conditions for both problems.

 It is worth noting that the mean field game  and mean field control  problems are closely related but have different objectives. 
In MFG, the objective is to achieve an equilibrium where no agent can improve their reward by unilaterally changing their strategy. In contrast, MFC focuses on maximizing the collective reward of all agents under a centralized policy. Remark \ref{rem-MFG-MFC} further elaborates on this distinction.

{  In \eqref{e:sec2game} and \eqref{e:sec2control}, the reward functional $J$ depends on the long-term average supply rate $\kappa^Q$ or $\kappa^R$.  One may consider alternative formulations of the MFG and MFC problems, in which the reward functional depends on the  supply rate up to the present time; c.f. \eqref{e:reward-hatJ}.  However,  Remark \ref{Rem:t-formulation} explains that such formulations are essentially equivalent to the ones given in \eqref{e:sec2game} and \eqref{e:sec2control} under the long-term average criterion.}

When solving the MFG problem, a key step in the analysis is to fix a value of $z$ (and hence a price $p = \varphi(z)$) and study the corresponding classical long-term average optimal impulse control problem.  This is fully solved in the paper \cite{HelmSZ:26} using a renewal argument to explicitly represent the long-term average reward for the class of thresholds policies as a nonlinear function; see \defref{thresholds-policy} for the class of policies and \eqref{e:F_K} for the function. For convenience of presentation, we recall some key results in Proposition \ref{prop-Fmax}, which establishes the optimality of a thresholds policy under certain conditions.  The solution technique relies on the boundary structure and the speed measure and scale function corresponding to a one-dimensional diffusion.

Turning to the MFG  problem, {  a challenge in leveraging the optimality results for the classical impulse control
problem is that the price $ p = \vphi(\kappa^{Q})$ depends on the long-term average product supply rate $\kappa^{Q}$ to the market. Thus the analysis must allow the price to vary as $\kappa^{Q}$ varies.  The central idea is that for each long-term average supply rate $z$, we solve the corresponding classical impulse control problem and determine the resulting long-term average supply rate. This leads to a continuous function of the supply rate, for which we establish the existence of a fixed point. The fixed point,   in turn,  yields a  thresholds-type  equilibrium impulse policy within the class of admissible policies in Theorem  \ref{prop-game-equi}.}

 For the MFC problem \eqref{e:sec2control}, we derive the key fixed-point-type identity \eqref{e:z=Bid/Bxi}, which facilitates the establishment of an upper bound for the functional $J(R, \kappa^{R})$ for {\em all} admissible impulse policies $R$.   This identity depends on subtle asymptotic analysis of optimal thresholds for a family of auxiliary classical impulse control problems.  Next  it is shown  that a specific  thresholds policy achieves this upper bound, and thus is a mean-field optimal impulse policy; see Theorem \ref{thm-mean-field-control} for details. 
 
 We note that this approach differs from the probabilistic and analytic methods presented  in \cite{CarmoD-18I} and  \cite{BensFY:13}. {  In contrast, our approach in \cite{HelmSZ:26} exploits renewal theory to express the reward rates of thresholds policies as a nonlinear function of the thresholds.  Conditions are given for the existence and uniqueness of a maximizing pair of thresholds, which leads to the solution of the classical impulse control problem.  In this paper, the solution of the classical impulse control problem, together with the fixed-point argument, leads to the existence of an equilibrium for MFG that is a thresholds-type policy. For the MFC, after deriving the upper bound for all $J(R, \kappa^R)$, we derive an explicit solution through a direct method, resulting in an optimal thresholds policy.  However, it is important to emphasize that establishing the upper bound for  $J(R, \kappa^{R})$ over all admissible policies $R$ is far from straightforward. Its derivation hinges on the critical fixed-point-type identity \eqref{e:z=Bid/Bxi}.}

 The long-term average mean field game and control problems \eqref{e:sec2game} and \eqref{e:sec2control} are motivated by and are extensions of those in the paper \cite{Chris-21}. In  their  formulation, $c\equiv 0$ and an exogenous post-impulse level $y_{0}$ is given so that $X^{R}(\tau_{k}) = y_{0}$ for each $k\in \NN$. In addition, the left boundary $a$ is assumed to be an entrance boundary.  In our formulation, $a$ can be an entrance or a natural boundary, thus enlarging the applicability of the model. Moreover,   the post-impulse level $X^{R}(\tau_{k}) $ is not pre-determined and can be chosen, and the right boundary point $b$ can be infinite or finite. The latter case requires an additional condition to be imposed on how fast the diffusion moves close to $b$.
 
 In addition, it is important to differentiate our results from those in \cite{Chris-21}, which proves the existence of an equilibrium harvesting strategy among  single threshold policies and an optimal mean-field strategy among stationary policies. We substantially extend these findings by deriving equilibrium and optimal mean-field strategies within the set of {\em all} admissible impulse strategies.

The study of mean field games and mean field control has experienced a significant surge of interest in recent decades, sparked by the pioneering  works of \cite{LasryL-07} and \cite{HuangMC-06}. For comprehensive expositions on these topics, we refer the reader to \cite{CarmoD-18I, CarmoD-18II} and \cite{BensFY:13}.
In recent years, there has been growing attention to stationary and ergodic formulations of mean field games and control problems. Notably, the long-time behavior of such problems has been investigated in \cite{BardiK-24, CardM-21, CiraP-21}. We also refer to \cite{BayrK-24, BaoT-23, BernC-23, ArapBC-17, Feleqi-13, Albe-22} and  the references therein for recent progress in the study of ergodic mean field games and control.
Furthermore, ergodicity and turnpike properties in linear-quadratic mean field games and control problems have been explored in \cite{BayrJ:25, SunY-24}.

  Notably, the literature on mean field games and control in the context of impulse control remains relatively limited.   Beyond the aforementioned  \cite{Chris-21},  the work \cite{BaseCG-22}  develops and solves a discounted symmetric mean field game involving impulse controls.   We also refer  to the recent work \cite{CaoDF-23}, which analyzes stationary discounted and ergodic mean field games with singular controls. {  
This line of research is extended in \cite{DianFT:23}, where the authors study an ergodic mean field game problem involving two-sided singular controls for regime-switching diffusions. See also the recent paper \cite{ChrisMO-25}, which considers an ergodic mean field game for two-sided singular controls of one-dimensional diffusions. We note that these papers and the current work share several common features: the underlying state process is one-dimensional; the mean field interaction occurs through the distribution of the state process; the scale and speed measures play important roles in the analyses; the focus is on long-term average criteria; and the optimal policies are of barrier or threshold type. Nevertheless, there are also important  differences between these works and our study.

In contrast to the singular control framework adopted in \cite{CaoDF-23,DianFT:23}, and \cite{ChrisMO-25}, the present work focuses on MFG and MFC with impulse controls. These two types of controls are  different, making the analyses and results not directly comparable. In particular, the fixed cost associated with each intervention in our model leads to a qualitatively distinct structure of optimal policies, which are of thresholds-type, as opposed to the continuous adjustments associated with reflection policies derived from solutions to Skorokhod problems. Furthermore, the methods employed in this paper--leveraging renewal theory and the renewal reward theorem--differ markedly from those typically used in singular control problems, which often involve solving  variational inequalities or free boundary problems with smooth pasting.}
 
The rest of the paper is organized as follows.  Section \ref{sect-form-assump}  presents the precise model formulation and collects the key conditions used in the subsequent analysis. It also introduces the class of $(w,y)$-policies and summarizes the main results of the paper. For convenience of presentation, we also recall the main results of   classical impulse control of \cite{HelmSZ:26} in Section~\ref{sect:classical impulse}.

The mean field game problem is studied in Section~\ref{sec-equilibrium}.
 To utilize the results in  Section~\ref{sect:classical impulse} for the classical ergodic impulse control problem, we introduce the class $\La$ of feasible prices  in \eqref{e:set-Lambda-defn}. 
This enables a fixed-point argument for the mapping $\mathfrak z\circ\Psi\circ\varphi$ defined in \eqref{e-MFG-fxpt}, which, in turn, establishes the existence of an equilibrium $(w,y)$-policy.

Section~\ref{sect-mean-field-control} studies the mean field control problem inspired by the Lagrange multiplier method.  It demonstrates that an optimal mean field impulse control policy exists 
under appropriate conditions on the functions $\varphi$, $c$, and $\mu$. Moreover, the optimal policy, which is of $(w,y)$ type, is explicitly characterized.

 A stochastic logistic growth model with an unbounded state space and a population growth model in a stochastic environment with a bounded   state space are presented in Section \ref{sect:example} for illustration.   Some technical proofs are deferred to the Appendix \ref{sect-appendix}.

Throughout the paper, we use the notation that $\lan f,\pi \ran : = \int f d\pi$ if $f$ is a function and $\pi$ is a measure, as long as the integral $ \int f d\pi$ is well-defined. The indicator function of a set $A$ is denoted  by $I_{A}$.

\section{Formulation  and Preliminaries}\label{sect-formulation}

We establish the model under consideration and collect some key conditions that will be used in later sections of the paper in Section \ref{sect-form-assump}. The main result of this paper is presented in Theorem \ref{thm-main-MFGC}.  To facilitate presentation, we also summarize the results of classical impulse control of \cite{HelmSZ:26} in Section \ref{sect:classical impulse}.  {  Our analysis relies on the boundary behavior of the uncontrolled process, as well as the speed measure and scale function, and therefore the state process requires it to be a one-dimensional diffusion.}

\subsection{Formulation,  Assumptions,  and Main Result}\label{sect-form-assump} 
{\bf \em Dynamics.}\/ 
Let $\I: = (a, b) \subset \RR$ with $a> -\infty$ and $b\le \infty$. In the absence of interventions, the process $X_{0}$ satisfies \eqref{e:X0} and is a regular diffusion with state space $\I$.   The measurable functions $\mu$ and $\sigma $ are assumed to be such that a unique non-explosive weak solution to  \eqref{e:X0} exists; we refer to Section 5.5 of \cite{Karatzas-S} for details.  For simplicity, we assume $(\Omega, \F, \{\F_{t}\}, \P)$ is  a filtered probability space with an $\{\F_{t}\}$-adapted Brownian motion $W$ and on which $X_0$ is defined, as well as each controlled process.  In addition, we assume that  $\sigma^{2}(x) > 0$ for all $x \in \I$.  We closely follow the notation and terminology on boundary classifications of one-dimensional diffusions in Chapter 15 of \cite{KarlinT81}.  The following standing assumption is imposed on the model throughout the paper:

\begin{cnd} \label{diff-cnd}
\begin{description}
\item[(a)] Both the speed measure $M$ and the scale function $S$ of the process $X_0$ are absolutely continuous with respect to Lebesgue measure. The scale and speed densities, respectively, are given by 
\begin{equation}
\label{e:s-m}
s(x) : = \exp\bigg\{-\int_{x_{0}}^{x} \frac{2\mu(y)}{\sigma^{2}(y)}dy \bigg\}, \quad m(x) = \frac{2}{\sigma^{2}(x) s(x)}, \quad x\in (a,b),
\end{equation} where $x_{0}\in \I$ is as in \eqref{e:X0} and is an arbitrary point, which will be held fixed. 
\item[(b)]  The left boundary $a > -\infty$ is a non-attracting point and the right boundary $b \le \infty$ is a natural point. 
 Moreover,  \begin{equation} \label{e:M(a-y)-finite}
  M[a, y] < \infty \text{ for each }y \in \I, 
\end{equation} and the (hitting time) potential function $\xi$ defined by  \begin{align}\label{e-xi}
 \xi(x) :=  \int_{x_{0}}^{x} M[a, v] dS(v), \quad  x\in \I
\end{align}  satisfies
\begin{equation}\label{e-sM-infty}
    \lim_{x\to b} \xi'(x) = \lim_{x\to b} s(x) M[a,x] =\infty.
  \end{equation}
\item[(c)] The function $\mu$ is continuous on $\I$ and extends continuously to the boundary points  with $|\mu(a)| < \infty$.
\end{description}
\end{cnd} 

 Condition~\ref{diff-cnd}(a) places restrictions on the model \eqref{e:X0} which seem quite natural for harvesting problems and other applications, such as in mathematical finance.
 The assumption that $a> -\infty$ is a non-attracting point implies that it cannot be attained in finite time by the uncontrolled diffusion.  For growth models  with $a=0$, this condition excludes the possibility of extinction. Likewise,  $b\le \infty$ being a natural boundary prevents the state from exploding to $b$ in finite time.  Note that $a$ can be either an entrance point or a natural point; the state space for $X_{0}$ is respectively $\E = [a, b)$ or $\E = (a, b)$.    
 
\cndref{diff-cnd}(b,c) imposes further limitations on the model.  For instance, the assumption that $|\mu(a)| < \infty$ excludes the consideration of Bessel processes. The assumption that $a$ is non-attracting further implies that $\mu(a)\ge 0$ and that $a$ is an entrance point if $\mu(a) > 0$.  In addition, the finiteness condition \eqref{e:M(a-y)-finite} always holds when $a$ is an entrance boundary but eliminates some diffusions when $a$ is natural;  see Table~6.2 on p.~234 of \cite{KarlinT81}.  Moreover, this requirement implies that the expected hitting times from $w$ to $y$ are finite whenever $a<w<y<b$. 

We now specify the class of admissible impulse policies which, apart from the transversality condition in (iv)(b), is quite standard.  
\begin{defn}[Admissibility]  \label{admissible-policy}
We say that $R : =\{(\tau_{k}, Y_{k}), k =1,2,\dots\}$ is an {\em admissible impulse policy}\/ if 
\begin{itemize}
 \item[(i)] $\{\tau_{k}\}$ is an increasing sequence of $\{\F_{t}\}$-stopping times with $\lim_{k\to\infty} \tau_{k} =\infty$, 
  \item[(ii)]  for each $k\in \NN$,  $Y_{k}$ is $\F_{\tau_{k}}$-measurable with $0 < Y_k \leq X^R(\tau_k-) - a$ when $\tau_k < \infty$, where equality is only allowed when $a$ is an entrance boundary; 
  \item[(iii)] $X^R$ satisfies \eqref{e:X} and we set $\tau_{0} =0$ and $X^{R}(0-) = x_{0}\in \I$; and  
 \item[(iv)]  {\em if} $a$ is a natural boundary, either 
 \begin{itemize}
 \item[(a)] there exists an $N \in \mathbb N$ such that $\tau_N =\infty$, which implies $\tau_k = \infty$ for all $k \geq N$ and, to completely specify the policy, we set $Y_k = 0$ for all $k \ge N$; or
 \item[(b)] $\tau_k < \infty$ for each $k \in \mathbb N$ and, for the function $\xi$  defined in \eqref{e-xi}, it holds that 
 \begin{equation} 
   \label{eq-xi-transversality}
    \lim_{t\to \infty} \limsup_{n\to \infty}t^{-1} \EE[\xi^{-}(X^R(t\wedge \beta_{n}) )]  = 0,
   \end{equation}
 where $\xi^-$ denotes the negative part of the function $\xi$, and  for each $n\in \NN$,  $\beta_n := \inf\{t  \geq 0: X^R(t) \notin( a_n,  b_n)\}$, in which   $\{a_n\}\subset \I $ is a decreasing sequence with $a_n \to a$ and $\{b_n\}\subset \I $ is an increasing sequence with $b_n \to b$.  
\end{itemize} \end{itemize}
 We denote by $\A$ the set of admissible impulse strategies. 
\end{defn} 

\begin{rem} \label{rem-admissible-policy}
  An admissible impulse policy $R$ satisfying Definition \ref{admissible-policy} (i), (ii), (iii), and (iv)(a) has a finite number of  interventions or no intervention (corresponding to the case when $\tau_1 =\infty$); the latter  is  called a ``do-nothing'' policy and is denoted by $\mathfrak R$.  For convenience of later presentation, we denote by $\AF$ the set of admissible policies with  at most a finite number of interventions.   
 
 On the other hand, if $R\in \A$ satisfies Definition \ref{admissible-policy} (i), (ii), (iii), and (iv)(b), the number of interventions is infinite; the set of such  policies is denoted by $\AI$. We have 
 $$\A =  \AF \cup \AI \quad \text{  and }\quad  \AF \cap \AI = \emptyset.$$  
 Note that \eqref{eq-xi-transversality} is a transversality condition {\em imposed only on diffusions for which $a$ is a natural boundary.}  It is satisfied by the $(w,y)$-policies to be defined in \defref{thresholds-policy}; see  Proposition \ref{prop-Rwy} for details.  When $a$ is an entrance boundary, \eqref{eq-xi-transversality} is automatically satisfied because $\xi$ is bounded below. Finally we point out that the $\{ \beta_n\}$ sequence in \eqref{eq-xi-transversality} satisfies   $\lim_{n\to\infty} \beta_{n} =\infty$ a.s. since  $a$ is non-attracting and $b$ is natural thanks to Condition \ref{diff-cnd}.  
 
  We emphasize that the admissible impulse policies defined in Definition \ref{admissible-policy} are not required to be of any particular type, such as a thresholds policy or a stationary policy.  For example, while the class of thresholds policies belongs to the admissible set  $\A$, nonstationary policies that alternate  between a finite number of such policies are also admissible.  Requiring transversality in \defref{admissible-policy}(iv)(b) for models in which $a$ is a natural boundary is a weak condition that allows for a large class of admissible policies.
  \end{rem}
  
We next introduce an important special class of impulse policies that will play a central role in the analyses of both the MFG and MFC problems. 
\begin{defn}[$(w,y)$-Policies] \label{thresholds-policy}
 Let $(w,y) \in \cR: = \{(w,y)\in \E\times \E: w < y\}$ and set $\tau_0 = 0$ and $X^{(w,y)}(0-)=x_0$.  Define the $(w,y)$-policy $R^{(w,y)}$, with corresponding state process $X^{(w,y)}$, such that for $k \in \NN$,
$$ \tau_k = \inf\{t > \tau_{k-1}: X^{(w,y)}(t-) \geq y\} \qquad \text{and} \qquad Y_k = X^{(w,y)}(\tau_k-) - w.$$ 
The   definition of $\tau_k$ must be slightly modified when $k = 1$ to be $\tau_1 = \inf\{ t\ge 0: X(t-)\ge y\}$ to allow for the first jump to occur at time 0 when $x_0\ge y$.

Under this policy, the impulse controlled process $X^{(w,y)}$ immediately resets to the level $w$ at the time it would reach (or initially exceed) the threshold $y$. \end{defn} 

\begin{rem}\label{Rem-about-(wy)}
 {  $(w, y)$-policies are also referred to as threshold policies in the literature. They are important in applications due to their intuitive structure and ease of implementation. From a theoretical perspective, the controlled processes under such policies are more tractable to analyze; in particular, renewal theory can be employed to study their long-term behavior. Moreover, $(w,y)$-policies are often optimal for a broad class of impulse control problems; see,  for example, \cite{Bens:05, DaiY-13-average, HeYZ-17,HelmSZ-17,HelmSZ:26}. In addition, as established in Theorems \ref{prop-game-equi} and \ref{thm-mean-field-control}, an equilibrium policy for the MFG problem \eqref{e:sec2game} and an optimal policy for the MFC problem \eqref{e:sec2control} are of $(w,y)$-type.} 
\end{rem}

The next proposition establishes an important bound $z_0$ on the product supply rate for any $(w,y)$-policy.  In addition, for each $z \in (0,z_0)$, it shows that there is a $(w,y)$-policy having supply rate $z$. To preserve the flow of the presentation, we defer the proof of this proposition to Appendix~\ref{sect-appendix}.  This result is crucial for the analyses of both the MFG and MFC problems.   Define the function   \begin{align}\label{eq-kappa-Q}  {\mathfrak z}(w,y): = \frac{y-w}{\xi(y) - \xi(w)}, \ \quad  (w, y) \in \cR. \end{align} 

\begin{prop}  \label{prop-z=Bid/Bxi} Assume Conditions \ref{diff-cnd} holds.   
Then
\begin{description}
  \item[(i)]  the function $\xi$ of \eqref{e-xi} satisfies  
\begin{equation}
\label{eq:z0defn}
 z_{0}: = \sup_{x\in \I} \frac{1}{|\xi'(x)|} = \sup_{x\in \I} \frac{1}{\xi'(x)} <  \infty;
\end{equation}
\item[(ii)]  for any $z\in (0, z_{0})$, there exists a pair     $  (w,  y) \in \cR$ so that $\mathfrak z(w,y) = z$; 
  \item[(iii)]  on the other hand, $\mathfrak z(w,y) \le z_{0}$ for every  $  (w,  y) \in \cR$; and  
  \item[(iv)]  if Condition~\ref{3.9-suff-cnd} also holds,  then $\mathfrak z(w,y) <  z_{0}$ for every  $  (w,  y) \in \cR$.  
\end{description} 
\end{prop}

We now turn to the formulation of the rewards.

\noindent
{\bf \em Reward Structure.}\/  A running reward is earned at rate $c$, which depends on the state of the process $X^{R}$.  The impulse reward is proportional to the size of the impulse,  with the unit price determined by the market's overall supply rate through the price function $\vphi$.  Each intervention also incurs a fixed cost $K> 0$. Consequently, given the supply rate $z$ and corresponding price $p=\vphi(z)$, the long-term average reward for the product manager who adopts the strategy $R\in \A$ is given by \eqref{e:reward-fn}. We assume that $c$ and  $\vphi$ satisfy the following condition.

\begin{cnd} \label{c-cond}  \begin{description} \item[(a)]   
 The function $c:\I\mapsto \R_{+}$ is continuous,     increasing, and extends continuously at the endpoints, with $0 \le c(a) < c(b) < \infty$.  
 \item[(b)] The function $\vphi: \R_{+}\mapsto \R_{+} $ is continuous. {  Moreover, 
 for each supply rate $z\in [0, z_{0}]$ and the corresponding price $\vphi(z)$, there exists some $(w,y)$-policy $R^{(w,y)}$ (defined in \defref{thresholds-policy}) that outperforms the  ``do-nothing'' policy $\mathfrak R$. That is,  for the  reward function $J(R^{(w,y)};z)$ defined in \eqref{e:reward-fn}, we have  
\begin{equation}
\label{eq:fund-assump}
J(R^{(w,y)};z) > J(\mathfrak R) = \liminf_{t\to \infty}t^{-1} \EE\bigg[ \int_0^t c(X_0(s))\, ds\bigg].
\end{equation}
The right-most   expression of \eqref{eq:fund-assump} is the long-term average running reward in the absence of any interventions, as indicated using the process $X_0$.} 
\end{description}
\end{cnd} 

\begin{rem} Several comments are in order regarding Condition \ref{c-cond}. 

The mean field problems become simpler when the running reward $c$ is constant; this constant rate merely adds to the net long-term average reward rate from impulse interventions.  
Our analysis remains valid, and the expressions simplify significantly when $c$ is constant.  We present the more challenging problems in which $c$ is non-constant.

As indicated in the introduction, to have meaningful mean field game and control problems we require that the price structure be such that it is always beneficial to intervene with an impulse as compared to the do-nothing policy.  Condition \ref{c-cond}(b) is therefore imposed on the model.   {   Remark \ref{rem-about-condition2.6}(a) presents an equivalent formulation  for this condition   using the notion of   ``feasible prices."   Furthermore, should this condition not be satisfied, then the do-nothing policy $\mathfrak R$ may be both an equilibrium for the MFG \eqref{e:sec2game} and an optimal policy for the MFC \eqref{e:sec2control}; see Remarks \ref{rem-about-condition2.6}(b) and \ref{rem-MFC-cond26} for details.  }

Finally we note that     
 the long-term average reward for every $R\in  \AF$ is equal to that of the do-nothing policy  $\mathfrak R$. In addition, we have $\kappa^R =0$ for every $R \in \AF$.
\end{rem}

We now introduce a key technical condition necessary for developing the fixed-point arguments in Sections \ref{sec-equilibrium} and \ref{sect-mean-field-control}; it connects the mean growth rate $\mu$ to the running reward  rate $c$.   Together with Conditions \ref{diff-cnd} and   \ref{c-cond}, 
 this condition allows us to establish and characterize the existence of a mean-field equilibrium in \thmref{prop-game-equi} and a mean-field optimal impulse control in Theorem \ref{thm-mean-field-control}.

\begin{cnd} \label{3.9-suff-cnd}
There exists some $\hat x_{\mu,c}\in \I$ so that $\mu$ is strictly increasing on $(a, \hat x_{\mu,c})$ and the functions $c$ and $\mu$ are concave on $(\hat x_{\mu,c}, b)$.
\end{cnd}

{  The next theorem  presents the main result of this paper. 
\begin{thm}\label{thm-main-MFGC}
  Suppose Conditions \ref{diff-cnd}, \ref{c-cond}, and \ref{3.9-suff-cnd} hold. Then both the MFG \eqref{e:sec2game} and the MFC \eqref{e:sec2control} are solvable in $\A$. In other words, there exist $Q^{\mathrm e}, Q^\star \in \A $ so that \begin{align*} 
    J(R, \kappa^{Q^{\mathrm e}}) & \le J(Q^{\mathrm e}, \kappa^{Q^{\mathrm e}}), \quad \forall R\in \A,\intertext{and}  
    J(R, \kappa^R) & \le J(Q^\star, \kappa^{Q^\star}), \quad \forall R\in \A.
  \end{align*} 
    $Q^{\mathrm e}$ and $Q^\star$ are referred to as the {\em mean field equilibrium} and {\em optimal mean field control} policies, respectively, and both are of $(w,y)$-type.  
\end{thm} \begin{proof}
  See Theorems \ref{prop-game-equi} and  \ref{thm-mean-field-control}.  
\end{proof}}

  To facilitate subsequent  presentations, we   introduce the   operators $A$ and $B$ as well as the average expected occupation and impulse measures that will be used frequently in the sequel.
The generator of the process $X^{R}$ between jumps (corresponding to the uncontrolled diffusion process $X_{0}$) is 
\begin{equation} \label{generator}
A f:= \frac12 \sigma^{2} f'' + \mu f' = \frac12 \frac{d}{dM}(\frac{df}{dS}),
\end{equation}  
where $f \in C^{2}(\I)$.   
For any function $f: \E  \mapsto \R$, define  
 \begin{equation} \label{e:Bf-defn}
  Bf(w, y): = f(y) - f(w), \quad (w, y) \in \cR.
\end{equation} The operator $B$ captures the effect of an instantaneous impulse. 

For any policy $R\in\mathcal A$, for each $t > 0$, we can define the average expected occupation measure $\mu^{R}_{0,t}$ on ${\cal E}$ and the average expected impulse  measure $\mu^{R}_{1,t}$ on $\overline{\cal R}$ by
\begin{equation} \label{mus-t-def}
\begin{aligned}
\mu^{R}_{0,t}(\Gamma_0) &= \displaystyle  t^{-1} \EE\left[\int_0^t I_{\Gamma_0}(X^{R}(s))\, ds\right], & \quad \Gamma_0 \in {\cal B}({\cal E}), \rule[-15pt]{0pt}{15pt} \\
\mu^{R}_{1,t}(\Gamma_1)  &= \displaystyle  t^{-1} \EE\left[\sum_{k=1}^\infty I_{\{\tau_k \leq t\}} I_{\Gamma_1}(X^{R}(\tau_k),X^{R}(\tau_k-))\right], & \quad \Gamma_1 \in {\cal B}(\overline{\cal R}),
\end{aligned}
\end{equation} 
Using these measures, we can rewrite the long-term average product supply rate $\kappa^{R}$ of \eqref{e:kappaQ} as \begin{equation}
\label{e2:kappaQ}
\kappa^{R} =  \limsup_{t\to\infty} \lan B\id, \mu^{R}_{1, t}\ran,
\end{equation} where $\id(x) : =x, x\in \E$, is the identity function.  Moreover, the functional $J(R; z)$ of \eqref{e:reward-fn} can be expressed as 
\begin{equation}
\label{e2:reward-fn}
J(R; z) =  \liminf_{t\to\infty} [ \lan c, \mu^{R}_{0,t}\ran + \lan \vphi(z) B\id     - K, \mu^{R}_{1,t}\ran].
\end{equation}

Our solution to the mean-field game and impulse control problems \eqref{e:sec2game} and \eqref{e:sec2control} relies on detailed analyses of $(w,y)$-policies introduced in Definition \ref{thresholds-policy}.  \propref{prop-Rwy} collects some important properties of these policies, while  Proposition \ref{prop-z=Bid/Bxi} analyzes their associated supply rates. 
In addition to the time potential $\xi$ of \eqref{e-xi}, 
we define the running reward potential $g$: \begin{align}
\label{e:g-defn}
g(x)&: = \int_{x_{0}}^{x}\int_{a}^{v} c(u) dM(u) dS(v), \quad  x\in \I. \end{align} 
These functions will play a central role in the ensuing analysis for mean field game and mean field control problems. We refer to \cite{HelmSZ:26} for the key  properties of  these functions.


{  \begin{rem}\label{Rem:t-formulation} We conclude this subsection with a remark on the reward functional $J(R, \kappa^Q)$ 
given in \eqref{e:sec2game} and \eqref{e:sec2control}, where $Q, R \in \A$. In view of \eqref{e:reward-fn} or \eqref{e2:reward-fn}, we see that  the reward obtained from the impulse control depends on the long-term average supply rate 
$\kappa^Q$ through the price function $\vphi$. One may consider an alternative reward functional where the price depends on the  average supply rate up to {\em the present time}.  In other words, letting $ \kappa_t^Q: = \lan B\id, \mu_{1,t}^Q\ran$ be the average supply rate up to time $t$ of the policy $Q\in \A$, then we can define \begin{align}\label{e:reward-hatJ}
  \hat J(R, Q) : = \liminf_{t\to\infty}[\lan c, \mu_{0,t}^R\ran + \lan \vphi({ \kappa_t^Q}) B\id - K, \mu^R_{1, t}\ran].
\end{align} 
 However,   there is no difference between these formulations as long as $Q\in \A_0$, where $\A_0: = \{R\in \A: \kappa^R = \lim_{t\to\infty} \kappa_t^R\}$. Note that $\A_0 \neq \emptyset$ because it contains all $(w,y)$-policies as well as all policies in $\AF$ and policies which alternate between finitely many $(w,y)$-policies.   In other words, for any $R\in \A $ and $Q\in \A_0$, we have 
 \begin{align}\label{eq:reward_equiv}
  J(R, \kappa^Q) = \hat J(R, Q).
 \end{align} 
 The proof of \eqref{eq:reward_equiv} is arranged in Appendix \ref{sect-appendix}.
\end{rem}  
  }


\subsection{Classical Impulse Control Problem}\label{sect:classical impulse}

We now summarize our solution in \cite{HelmSZ:26} to   the classical long-term average impulse control problem  of maximizing the reward functional
\begin{align}
\label{e:reward-1player} 
\nonumber
J(R) : &= \liminf_{t\to\infty}  t^{-1} \EE \bigg[ \int_{0}^{t}c(X(s)) ds +  \sum_{k=1}^{\infty} I_{\{\tau_{k} \le t\}}  [p (X(\tau_{k}-)- X(\tau_{k}) ) -   K]\bigg]\\ 
  & =  \liminf_{t\to\infty} [\lan c, \mu_{0, t}\ran + \lan pB\id -K, \mu_{1, t}\ran], 
\end{align}
where $R = \{ (\tau_{k}, Y_{k}), k =1, 2,\dots\}\in \A$ is an admissible impulse policy, $X= X^{R}$ is the controlled process with initial condition $X(0-) = x_{0}\in \I$,  $p>0$ is the unit price for impulse, and $   K > 0$ is the fixed cost for each execution of impulse.  Recall the measures $ \mu_{0, t} =   \mu_{0, t}^{R}$ and $ \mu_{1, t}=  \mu_{1, t}^{R}$ defined in \eqref{mus-t-def}. For notational simplicity, we omit the superscript $R$ in $X^{R},   \mu_{0, t}^{R}$, and  $\mu_{1, t}^{R}$  throughout the section. Note that the scaling parameter $\gamma$ of \cite{HelmSZ:26} is set to be $\gamma=1$ in this section.  

First, the model satisfies \cndref{diff-cnd} and the admissible policies are defined in \defref{admissible-policy}.  \cndref{c-cond}(a) on the subsidy $c$ is also imposed but \cndref{c-cond}(b) simplifies due to $p$ being fixed; see \cndref{cond-interior-max} below. All results in this section, except for Lemma \ref{lem-h-new}(ii), are established in \cite{HelmSZ:26}.   
 


We begin by establishing some important facts about $(w,y)$-policies, for $(w,y) \in \cR$.
\begin{prop}\label{prop-Rwy}
  Suppose Conditions \ref{diff-cnd} and \ref{c-cond}(a) hold. Then for any $(w,y) \in \cR$,
\begin{itemize}
  \item[{\em (i)}] the policy $R^{(w,y)}$ is an admissible impulse policy in the sense of Definition \ref{admissible-policy}; 
  \item[{\em (ii)}]  the long-term average supply rate of the policy $R^{(w,y)}$ is 
\begin{align*}
\kappa^{R^{(w,y)}} & 
  =  \limsup_{t\to\infty}   \lan B\id, \mu^{R^{(w,y)}}_{1,t}\ran 
= \mathfrak z(w,y),    
\end{align*}  
  where the function $\mathfrak z$ is defined in \eqref{eq-kappa-Q}; and 
\item[{\em (iii)}]  $J(R^{(w,y)};z)
= F_{p}(w, y)$, where  $p = \vphi(z)$ and  the function $F_p$ is defined by 
 \begin{equation}
\label{e:F_K}
F_{p}(w, y): = \frac{  g(y) - g(w) + p(y-w)- K}{\xi(y) - \xi(w)}, \quad  (w,y) \in \mathcal R.
\end{equation}
\end{itemize}
\end{prop} 

 Two additional conditions are imposed on the model for the classical impulse control problem.  We first identify the value of the do-nothing policy $\mathfrak R$  since the replacement for \cndref{c-cond}(b) assumes the existence of some $(w,y)$-policy  that outperforms  $\mathfrak R$. 
Under $\mathfrak R$, the corresponding (un)controlled process is $X_0$ of \eqref{e:X0} and  \cite{HelmSZ:26} demonstrates   that  
\begin{align}\label{e:reward-zero-policy}
 J(\mathfrak R) & = \liminf_{t\to \infty}t^{-1} \EE\bigg[ \int_0^t c(X_0(s))\, ds\bigg]  = \bar c(b), 
\end{align}  where   \begin{equation} \label{eq:c-bar(b)}
   \bar c(b): =  \begin{cases}
c(b),     & \text{ if } M[a, b] = \infty,\\  
\frac{\lan c,M\ran}{M[a, b]},      & \text{ if } M[a, b] < \infty.  \end{cases} 
\end{equation}    

We now introduce a simplification of \cndref{c-cond}(b).

\begin{cnd}\label{cond-interior-max}
There exists a pair $(\wdt w_{p}, \wdt y_{p}) \in \cR$ so that $F_{p} (\wdt w_{p}, \wdt y_{p}) > \bar c(b)$. 
\end{cnd}   

 For the classical impulse control problem, a weaker condition than \cndref{3.9-suff-cnd} is imposed on the model which is used to establish the uniqueness of a maximizer for $F_p$.  The condition involves the interplay, at price $p$, between the subsidy rate $c$ and the mean growth rate $\mu$.  
 Define the reward function \begin{equation}\label{e:r_p}
  r_p(x) = c(x) + p \mu(x), \ \ x\in \I.
\end{equation} \begin{cnd}  \label{r_p-behavior}
There exist some $\wdh x_p, \wdt x_p \in \I$ with $\wdh x_p \leq \wdt x_p$  such that $r_p$ is strictly increasing on $(a,\wdh x_p)$, $r_p$ is constant on $[\wdh x_p, \wdt x_p]$ and $r_p$ is strictly decreasing on $(\wdt x_p , b)$.
\end{cnd} 
The first-order necessary conditions for the maximizer of $F_p$ are expressed in \eqref{e-1st-order-condition} following a rearrangement and involve the function $h_p$, defined by 
\begin{equation}  \label{e-h-fn-defn}
h_{p}(x) = \frac{g'(x) +p}{\xi'(x)},  \qquad x\in \I.
\end{equation}

The following proposition identifies the solution to the impulse control problem with fixed price $p$.
\begin{prop}\label{prop-Fmax} Assume Conditions \ref{diff-cnd}, \ref{c-cond}(a), \ref{cond-interior-max}, and \ref{r_p-behavior} 
  hold.
 \begin{itemize}
    \item[{\em(a)}] 
 There exists a  unique pair $(w_{p}^*,y_{p}^*)\in \cR$ so that \begin{equation}
\label{e-Fmax}
F_{p}(w^{*}_{p}, y^{*}_{p}) = \sup_{(w, y)\in \cR} F_{p}(w, y)=:F_p^*. 
\end{equation} Furthermore, the optimizing pair satisfies the following (rearranged) first-order conditions: 
\begin{itemize}
  \item[{\em(i)}] If $a$ is a natural point, then every optimizing pair  $(w_{p}^*,y_{p}^*) \in \cR$  satisfies $a < w^{*}_{p} < y^{*}_{p} < b$ and
\begin{equation}\label{e-1st-order-condition}  
F_{p}^* =  h_{p} (w^{*}_{p}) = h_{p}(y^{*}_{p}),
\end{equation} where the function $h_p$ is defined by \eqref{e-h-fn-defn}.

  \item[{\em(ii)}] If $a$ is an entrance point, then an optimizing pair  $(w_{p}^*,y_{p}^*) \in \cR$ may have $w_{p}^*=a$; in such a case, we have 
 \begin{equation}\label{e2-1st-order-condition}  h_{p} (a) \ge   {F_{p}(a, y^{*}_{p})} = F_p^* = h_{p}(y^{*}_{p}). \end{equation} 
But if $w_{p}^*>a$, \eqref{e-1st-order-condition} still holds.  
\end{itemize}  
\item[{\em(b)}] For any admissible policy
 $R\in {\mathcal A}$, 
  we have
\begin{align}\label{e-reward-bdd}\nonumber
J(R) & = \liminf_{t\to\infty} [\lan c, \mu_{0, t}\ran + \lan pB\id -K, \mu_{1, t}\ran] \\ 
               & \le  \limsup_{t\to\infty} [\lan c, \mu_{0, t}\ran + \lan pB\id -K, \mu_{1, t}\ran]  \le  F_{p}^{*} = F_{p}(w^{*}_{p}, y^{*}_{p}), 
\end{align} 
and the $(w^{*}_{p}, y^{*}_{p})$-strategy is an optimal policy. 
\end{itemize}
  \end{prop} 

  \begin{rem}\label{rem1-transversality}   
 In fact,  the $(w^{*}_{p}, y^{*}_{p})$-thresholds policy is optimal in a larger  class $\A_p\supset \A$, which we now define. Using the time and running reward potentials $\xi$ and $g$, we define the impulse reward potential $G_{p}$ by\begin{align*}
  G_p(x): = F_p^* \xi(x)- g(x), \quad x\in\I.
 \end{align*} 
 Next, we present an alternative to  (iv)(b) in Definition  \ref{admissible-policy}:
\begin{description}
\item[\defref{admissible-policy}(iv)(c)] $\tau_k < \infty$ for all $k\in \mathbb N$ and, {\em if} $a$ is a natural boundary, then 
\begin{equation}
  \label{eq-transversality}
  \liminf_{t\to \infty} \liminf_{n\to \infty}t^{-1} \EE[G_{p}(X(t\wedge \beta_{n}))] \ge  0
  \end{equation}  in which the sequence $\{\beta_n\}$ is given in \defref{admissible-policy}.
\end{description}
Now $\A_p$ is the class of admissible policies   satisfying \defref{admissible-policy}(i), (ii), (iii), and (iv)(a)  or (iv)(c).

We demonstrated in \cite{HelmSZ:26} that \eqref{eq-transversality} holds whenever \eqref{eq-xi-transversality} does, so $\A \subset \A_p$. In particular, the  $(w^{*}_{p}, y^{*}_{p})$-policy is optimal  in this larger class $\A_p$.
\end{rem}

  We finish this section with some important observations about the function $h_p$ defined in \eqref{e-h-fn-defn}. First, detailed  calculations  
reveal that for any $x\in \I$,  we have
 \begin{align} \label{sect 2-e-h-defn}  
 h_{p}(x)   
 = \frac{\int_{a}^{x} (c(u) + p \mu(u) )dM(u) }{M[a,x]}= \frac{\int_{a}^{x}   r_{p}(y)  dM(y)}{M[a, x]},
\end{align}
 and 
\begin{align}\label{e:h'-expression}  h'_{p}(x) &=  \frac{m(x)}{M[a,x]} (r_{p}(x) - h_{p}(x))  = \frac{m(x)}{M^{2}[a,x]} \int_{a}^{x} [r_{p}(x) - r_{p}(y)] dM(y).
\end{align}   
In addition, we have the following lemma. To preserve the flow of the paper, we present the proof of this lemma in Appendix \ref{sect-appendix}; the key ideas are similar to those used in the proof of Lemma 4.7 of \cite{HelmSZ:26} but with some important modifications as the assumptions in the two papers differ.  
 \begin{lem}\label{lem-h-new} 
  \begin{description}
    \item[(i)]  Assume Condition \ref{c-cond}(a). Then
 $\lim_{x\to b}h_{p}(x ) = \bar c(b)$ for any $p\in \R$. 
    \item[(ii)]  Assume Conditions \ref{diff-cnd},    \ref{c-cond}(a), \ref{3.9-suff-cnd}, and \ref{cond-interior-max}  hold. Define \begin{equation}
\label{e-y-hat-p-defn}
y_{p} : = \min \{x \in \I: h'_{p}(x) = 0\}.
\end{equation} Then $h_{p}$ is strictly increasing on $(a, y_{p} )$ and strictly decreasing on $(y_{p}, b)$.
  \end{description}
 \end{lem}


%


\section{Mean Field Games}\label{sec-equilibrium} 
Building upon the preparatory work on classical long-term average impulse control presented in Section~\ref{sect:classical impulse}, we proceed to analyze the mean field game \eqref{e:sec2game}.  Following the usual  mean field game framework, and considering that our mean field structure \eqref{e:reward-fn} relies on the stationary supply rate $z$, we seek an admissible policy $Q$ whose associated supply rate satisfies the fixed-point condition. Specifically, given the stationary supply rate $\kappa^{Q}$ and hence the unit price $\vphi(\kappa^{Q})$, an individual agent aims to optimize her policy, such that the resulting long-term supply rate matches   $\kappa^{Q}$.
This section's main results in  \thmref{prop-game-equi}  establish sufficient conditions for the existence of an equilibrium within the set of admissible impulse controls $  \A$.  Furthermore, the equilibrium is a $(w,y)$-policy and is explicitly characterized.
 
 We begin by observing that, as shown in Lemma \ref{lem2-transversality} below, 
   the long-term average supply rate of  {\em any admissible policy}\/ $R$ is bounded above by the  constant $z_0$ defined in \eqref{eq:z0defn}
  due to the transversality requirement \eqref{eq-xi-transversality}.
 As a result, we can restrict    $\vphi$ to be a   function from $[0,  z_{0}] $ to $\R_{+}$, which explains the reason \cndref{c-cond}(b) is restricted to $z \in [0,z_0]$.

\begin{lem}\label{lem2-transversality} Assume Condition \ref{diff-cnd} holds.
For any $R=(\tau, Y) \in \A$, we have  \begin{align} \label{e:kappa<z0}
\kappa^{R} & =  \limsup_{t\to\infty} \lan B\id, \mu_{1,t}^{R}\ran  
\le z_{0} .
\end{align}
\end{lem}

\begin{proof} Obviously \eqref{e:kappa<z0} holds for $R\in \AF$ as $\kappa^R =0$. We now consider an arbitrary  $R  \in \AI$   and denote by $X= X^R$ the associated controlled process. Also let $\{\beta_{n}\}$ be  the sequence of stopping times given in Definition \ref{admissible-policy}. 
We apply It\^{o}'s formula to the process $\xi(X(t))$,  observing  that $A \xi(x) = 1$ for all $x\in \I$,  
\begin{align*} 
 \EE&_{x_0}[\xi(X(t\wedge \beta_{n}))]
  = \xi(x_0) + \EE_{x_0}[t\wedge \beta_{n}] + \EE_{x_0}\Bigg[  \sum_{k=1}^{\infty}I_{\{\tau_{k}\le t\wedge \beta_{n} \}} (\xi(X(\tau_{k})) - \xi(X(\tau_{k}-)))\Bigg].
\end{align*}
Since $\xi$ is a monotone increasing function, $\xi(X(\tau_{k})) - \xi(X(\tau_{k}-)) \le 0$ for each $k\in \NN$. Thus by the monotone convergence theorem, we have 
\begin{displaymath}
\lim_{n\to\infty} \EE_{x_0}[\xi(X(t\wedge \beta_{n}))] =  \xi(x_0) +  t +  \EE_{x_0}\Bigg[  \sum_{k=1}^{\infty}I_{\{\tau_{k}\le t \}} (\xi(X(\tau_{k})) - \xi(X(\tau_{k}-)))\Bigg]. 
\end{displaymath}
Dividing both sides by $t$ and  taking the limit  as $t\to\infty$,   we note that the limit on the left-hand side is nonnegative since  $R\in \AI$.  Thus,  we have 
\begin{equation}
\label{e2-transversality}
 \limsup_{t\to\infty}\;  t^{-1} \EE_{x_{0}}\Bigg[  \sum_{k=1}^{\infty}I_{\{\tau_{k} \le t\}} (\xi(X(\tau_{k}-)) - \xi(X(\tau_{k})))\Bigg] =  \limsup_{t\to\infty} \lan B\xi, \mu_{1, t}\ran \le   1.
\end{equation}  
Then, it follows from   \eqref{e2-transversality} and Proposition \ref{prop-z=Bid/Bxi} that   \begin{align*} 
\kappa^{R}  =  \limsup_{t\to\infty} \lan B\id, \mu_{1,t}^{R}\ran   =  \limsup_{t\to\infty} \lan \frac{B\id}{B\xi} B\xi, \mu_{1,t}^{R}\ran 
\le z_{0} <\infty.\end{align*} This establishes \eqref{e:kappa<z0} and hence completes the   proof.   
\end{proof}

To establish the existence of an equilibrium policy for \eqref{e:sec2game}, it is essential to analyze the behavior  of  $(w_{p}^{*}, y_{p}^{*})$ as $p$ varies, particularly for the case when $a$ is an entrance point. 
 We therefore   consider  the  following set: 
\begin{equation}
\label{e:set-Lambda-defn}
\La: = \{p\in \R: \text{ there exists some } (\wdt w_{p}, \wdt y_{p}) \in \cR \text{ so that } F_{p} (\wdt w_{p}, \wdt y_{p})  > \bar c(b)\}. 
\end{equation}   
In other words, $\La$ is the collection of ``feasible prices $p$''  satisfying Condition  \ref{cond-interior-max}. Note that for any $(w, y) \in \cR$, we have $\lim_{p\to\infty} F_{p}(w, y) = \infty$. This, together  with the assumption  that $c$ is bounded given in Condition \ref{c-cond}(a),  implies that $\La \neq\emptyset$. Moreover, it is obvious that if $p_{1} < p_{2}$ and $p_{1}\in \La$, then $p_{2} \in \La$. 
On the other hand, in view of \eqref{sect 2-e-h-defn},  Condition \ref{c-cond}(a),
and Lemma \ref{lem-h-new},
 for every $p\le 0$ and any $(w, y)\in \cR$, we have 
 \begin{displaymath}
F_{p}(w, y) < \frac{g(y) - g(w) }{\xi(y) - \xi(w)} =\frac{g'(x)}{\xi'(x)}= \frac{\int_{a}^{x} c(y) dM(y)}{M[a, x]} = h_{0}(x)< \bar c(b),
\end{displaymath} 
where $w < x < y$. 
Therefore $p \notin \La$  and hence we have $p_{0}: = \inf \La \ge 0$.  Moreover, for the case when $b$ is finite, using exactly the same argument as above, we can derive   that  $p_{0}   \ge \frac{K}{b-a}$. For convenience of later presentation, we summarize these observations in the following lemma: 
 \begin{lem}\label{lem-La set}
Under Conditions  \ref{diff-cnd} and  \ref{c-cond}(a),  the following assertions hold: 
\begin{enumerate}
  \item[{\em(i)}]  $\La \neq\emptyset$;
  \item[{\em(ii)}]  $p_{0} = \inf \La  \ge \frac{K}{b-a}$ $($with the understanding that $\frac K{b-a} =0$ if $b=\infty$$)$; and
  \item[{\em(iii)}] if $p_{1} < p_{2}$ and $p_{1}\in \La$, then $p_{2} \in \La$. 
\end{enumerate}
\end{lem}
\begin{rem}\label{rem-about-condition2.6}
 \begin{itemize}
  \item[(a)] We now  present an equilavent formuation of Condition \ref{c-cond}(b) using the set $\La$ of feasible prices defined in \eqref{e:set-Lambda-defn}. 
  On the one hand, if $\vphi_{\min}: = \min\{\vphi(z): z\in [0, z_{0}]\}\in \La$, then Lemma \ref{lem-La set}(iii) implies that $\vphi(z) \in \La$ for every $z\in [0, z_{0}]$. That is, there exists some $(\wdt w_{\vphi(z)}, \wdt y_{\vphi(z)}) \in \cR$ so that $F_{\vphi(z)}(\wdt w_{\vphi(z)}, \wdt y_{\vphi(z)}) > \bar c(b) $. In view of \eqref{e:F_K} and \eqref{e:reward-zero-policy}, this says that the long-term average reward of the $(\wdt w_{\vphi(z)}, \wdt y_{\vphi(z)})$-policy outperforms the do-nothing policy, yielding  \eqref{eq:fund-assump}.   

On the other hand, if   \eqref{eq:fund-assump} holds for every $z\in [0, z_{0}]$, using \eqref{e:F_K} and \eqref{e:reward-zero-policy} again, we have $\vphi(z) \in \La$ for every $z\in [0, z_{0}]$. This gives   $\vphi_{\min}\in \La$. 

{   Therefore,  Condition \ref{c-cond}(b) is equivalent to  the condition that $\vphi $ is continuous on $[0, z_{0}]$ and that $\vphi_{\min} = \min\{\vphi(z): z\in [0, z_{0}]\}\in \La$.}  

{  \item[(b)] Suppose that $\vphi(0) \notin \La$, then the do-nothing policy $\mathfrak R$ (or any $R \in \AF$) is an equilibrium for the MFG \eqref{e:sec2game}. Indeed, at one hand, we have $\kappa^{\mathfrak R} =0$ and  $J(\mathfrak R, \kappa^{\mathfrak R}) = \bar c(b)$. On the other hand, for any $R\in \A$, thanks to Proposition \ref{prop-Fmax}, we have $J(R, \kappa^{\mathfrak R}) \le F^*_{\vphi(0)}$; furthermore, the assumption that $\vphi(0) \notin \La$ impliest that $F^*_{\vphi(0)} \le \bar c(b)$.  This desmonstrates that $J(R, \kappa^{\mathfrak R}) \le J(\mathfrak R, \kappa^{\mathfrak R})$ for every $R\in \A$, and hence $\mathfrak R$ is an equilibrium for the MFG \eqref{e:sec2game}. 

In this case, the MFG \eqref{e:sec2game}   becomes trivial.  Condition \ref{c-cond}(b) rules out this trivial case and ensures that there exists at least one nontrivial equilibrium for the MFG \eqref{e:sec2game}.}\end{itemize}
\end{rem}

Under Conditions \ref{diff-cnd},   \ref{c-cond}(a), and \ref{3.9-suff-cnd}, for every $p\in \La$, there exists a unique pair  $(w_{p}^{*}, y_{p}^{*}) \in \cR$ so that $F_{p}(w_{p}^{*}, y_{p}^{*})  = \sup_{(w, y) \in \cR} F_{p}(w, y)$. 
 Consequently, we can consider the vector-valued function $\Psi: \La \mapsto \cR$ defined by
  \begin{equation}\label{e:Psi(p)-fn}
\Psi (p) := (w_{p}^{*}, y_{p}^{*})  = \argmax_{(w, y) \in \cR} F_{p}(w, y).
\end{equation}

\begin{prop}\label{prop-wp-yp-cont}
Assume Conditions \ref{diff-cnd},   \ref{c-cond}(a), and  \ref{3.9-suff-cnd} hold. 
Then the function $\Psi$ defined in \eqref{e:Psi(p)-fn} is continuous. 
\end{prop} 

\begin{proof}
We break the proof into two parts by considering the case of $w_p^* > a$ in the next lemma and $w_p^* = a$, which can only occur when $a$ is an entrance boundary, in \lemref{lem2-prop41-pf}.
\end{proof}

\begin{lem}\label{lem1-prop41-pf}
Assume Conditions \ref{diff-cnd},   \ref{c-cond}(a), and  \ref{3.9-suff-cnd} hold. Let $p\in (p_0,\infty)$ where $p_0 = \inf \La$. If $w_{p}^{*} > a$, then $\Psi$ is continuously differentiable at $p$.  
\end{lem}
\begin{proof}
We consider the function $\wdt{F}:   (p_0,\infty) \times \cR \mapsto \R^{2}$ defined by 
\begin{displaymath}
\wdt{F}(p,w,y): = \begin{pmatrix} h_{p}(w) - F_{p}(w, y) \\ h_{p}(y) - F_{p}(w, y)
              \end{pmatrix},
\end{displaymath} where the functions $F_{p} $ and $h_{p}$ are defined in \eqref{e:F_K} and \eqref{e-h-fn-defn}, respectively. Note that $\wdt{F}$ is continuously differentiable with  $\wdt{F}(p, w^{*}_{p}, y^{*}_{p}) =0$ and, thanks to \eqref{e-1st-order-condition}, we have 
\begin{align*}
\mathscr J& (p,  w_{p}^{*}, y_{p}^{*}):  = \begin{pmatrix} \partial_{w}\wdt F_{1} & \partial_{y}\wdt F_{1} \\  \partial_{w}\wdt F_{2} & \partial_{y}\wdt F_{2}\end{pmatrix}(p, w^{*}_{p}, y^{*}_{p}) = \begin{pmatrix}
h_{p}'(w_{p}^{*}) & 0 \\ 0 & h_{p}'(y_{p}^{*})
\end{pmatrix}.
\end{align*} 
 In view of the monotonicity of $h_{p}$ derived in 
Lemma \ref{lem-h-new}, we have $h'_{p}( w_{p}^{*} ) > 0$  and $h'_{p}( y_{p}^{*} ) < 0$.
Therefore $\mathscr J(p, w_{p}^{*}, y_{p}^{*})$ is an invertible matrix and hence we can apply the implicit function theorem to conclude that there exists an open   neighborhood $U$ containing  $p$ and a unique continuously differentiable function $\psi: U\mapsto \R^{2}$ such that $\psi(p) = ( w_{p}^{*}, y_{p}^{*}) = \Psi(p)$ and $\wdt F_{1}(x,\psi(x)) =0, \wdt F_{2}(x,\psi(x)) =0$ for $x\in U$.  In particular, this gives the continuous differentiability of $\Psi$ as desired. 
\end{proof}

When $a$ is an entrance point, it is possible that $w_{p}^{*} =a$ and $h_{p}(a) \ge F^{*}_{p} =  h_{p}(y_{p}^{*})$. Consequently we cannot directly apply the implicit function theorem to derive the continuity of the function $\Psi$ as in the proof of Lemma \ref{lem1-prop41-pf}.  To address this subtle issue, we begin by examining certain properties of   $F_{p}^{*}$ and the function $h_{p}$ when $p$ varies. 

\begin{lem}\label{lem-p-Fp*cont}
Suppose Conditions  \ref{diff-cnd} and \ref{c-cond}(a) hold. Let $$F_{p}^{*} : = F_{p}(w_{p}^{*}, y_{p}^{*})  = \sup_{(w, y) \in \cR} F_{p}(w, y).$$ Then the function $p\mapsto F_{p}^{*}, p \in \La$ is Lipschitz continuous with Lipschitz constant $z_0$: \begin{displaymath}
| F_{p_{2}}^{*} - F_{p_{1}}^{*}| \le z_{0}\, |p_{2}-p_{1}|, \quad \forall p_{1}, p_{2}\in\La.
\end{displaymath}   
\end{lem} \begin{proof}
Let $p_{1}, p_{2} \in \La$ with $p_{1} < p_{2}$ and $(w^{*}_{p_{2}}, y^{*}_{p_{2}})$ be a maximizing pair for the function $F_{p_{2}}$. Then we have \begin{align*} 
0 & \le   F_{p_{2}}^{*} - F_{p_{1}}^{*} \le F_{p_{2}}(w^{*}_{p_{2}}, y^{*}_{p_{2}})- F_{p_{1}}(w^{*}_{p_{2}}, y^{*}_{p_{2}}) \\
 &= (p_{2} -p_{1}) \frac{y^{*}_{p_{2}}-w^{*}_{p_{2}}}{\xi(y^{*}_{p_{2}})- \xi(w^{*}_{p_{2}})}= \frac{p_{2} -p_{1}}{\xi'(\theta)} \le z_{0}(p_{2} -p_{1}),
\end{align*} where $\theta \in (w^{*}_{p_{2}}, y^{*}_{p_{2}})$.  Thus the lemma is proved. 
\end{proof}
Using a similar argument, we have 
\begin{lem}\label{lem-p-hp-cont}
Suppose  Condition  \ref{diff-cnd} holds, then the function $p\mapsto h_{p}(x)$ is   Lipschitz continuous with constant $z_0$, uniformly in $x$: 
\begin{displaymath}
\sup_{x\in \I} |h_{p_{1}}(x) - h_{p_{2}}(x) | \le z_{0}\, |p_{1} - p_{2}|, \quad \forall p_{1}, p_{2} \in \R. 
\end{displaymath} 
\end{lem}

\begin{lem}\label{lem2-prop41-pf}
Assume Conditions \ref{diff-cnd},   \ref{c-cond}(a), and  \ref{3.9-suff-cnd} hold. Let $p\in \La$. If $w_{p}^{*} = a$, 
 then $\Psi$ is continuous at $p$.
\end{lem} \begin{proof} We use a contradiction argument. Recall 
 from Lemma \ref{lem-h-new}   that the function $h_{p}$ is strictly increasing on $(a, y_{p})$ and strictly decreasing on $(  y_{p}, b)$, where $y_{p} \in (a, b)$ is defined in \eqref{e-y-hat-p-defn}.
 Suppose $\Psi$ is not continuous at $p$, then there exists some $\e_{0} > 0$ so that for every $n\in \NN$, there exists some $p_{n}\in \La$ with $|p_{n}- p| < \frac1n$ so that  
\begin{displaymath}
|w^{*}_{p_{n}}- a| = w^{*}_{p_{n}}- a \ge \e_{0} \quad \text{ or }\quad |y^{*}_{p_{n}}- y^{*}_{p}| \ge \e_{0}.
\end{displaymath}  We can assume without loss of generality that $\e_{0} <y_{p}-a $ and $ \e_{0} <   b- y_{p}^{*}$ if $b$ is finite.

Let's first consider the case when $ w^{*}_{p_{n}}\ge a + \e_{0}$. Then using Lemma \ref{lem-h-new} again, we have 
\begin{align*} 
F^{*}_{p_{n}} & =h_{p_{n}}(w^{*}_{p_{n}})  \ge h_{p_{n}}(a + \e_{0})   \ge h_{p}(a + \e_{0}) - z_{0}|p_{n} -p| > h_{p}(a + \e_{0})  - \frac{z_{0}}{n},   
\end{align*} where the second inequality follows from Lemma \ref{lem-p-hp-cont}. On the other hand, since $w_{p}^{*} = a$, we have from \eqref{e2-1st-order-condition} that $F_{p}^{*}=h_{p}(y_{p}^{*})  \le h_{p}(a)$. Then it follows that  
\begin{displaymath}
F^{*}_{p_{n}} - F_{p}^{*} \ge  h_{p}(a + \e_{0})  - \frac{z_{0}}{n} -h_{p}(a)  > \frac{h_{p}(a + \e_{0})    -h_{p}(a) }{2} > 0, 
\end{displaymath} for all $n$ sufficiently large; note that the last inequality follows from Lemma \ref{lem-h-new} and the assumption that $a + \e_{0} < y_{p}$.    But this leads to a contradiction because  $ |F^{*}_{p_{n}} -F^{*}_{p}|  \to 0$ as $n\to \infty$ thanks to Lemma \ref{lem-p-Fp*cont}. 

We now consider the case when $|y^{*}_{p_{n}}- y^{*}_{p}| \ge \e_{0}$. Then in view of Lemma \ref{lem-h-new} and the choice of $\e_{0}$,  
we have three possible cases:  \begin{enumerate}
  \item[i)] if $y^{*}_{p_{n}} \ge y^{*}_{p} + \e_{0}$, then $|h_{p}(y^{*}_{p_{n}})- h_{p}( y^{*}_{p})| = h_{p}( y^{*}_{p})-h_{p}(y^{*}_{p_{n}}) \ge h_{p}( y^{*}_{p})- h_{p}( y^{*}_{p}+ \e_{0})> 0,$
  \item[ii)] if $y_{p} \le y^{*}_{p_{n}} \le y^{*}_{p}-  \e_{0}$, then $|h_{p}(y^{*}_{p_{n}})- h_{p}( y^{*}_{p})|= h_{p}(y^{*}_{p_{n}})- h_{p}( y^{*}_{p}) \ge h_{p}( y^{*}_{p}-  \e_{0}) -h_{p}( y^{*}_{p}) > 0,$ and 
  \item[iii)] if  $a < y^{*}_{p_{n}} < y_{p}$, then as $h_{p}(y^{*}_{p_{n}}) > h_{p}(a) \ge  h_{p}( y^{*}_{p})$, we have \begin{align*} 
|h_{p}(y^{*}_{p_{n}})- h_{p}( y^{*}_{p})| &= h_{p}(y^{*}_{p_{n}})- h_{p}( y^{*}_{p})\\ & =  h_{p}(y^{*}_{p_{n}})- h_{p_{n}} (y^{*}_{p_{n}}) +h_{p_{n}} (y^{*}_{p_{n}})- h_{p_{n}}(y_{p}) \\ & \quad + h_{p_{n}}(y_{p})  - h_{p}(y_{p})+h_{p}(y_{p})-  h_{p}( y^{*}_{p}) \\
 & >  h_{p}(y_{p})-  h_{p}( y^{*}_{p}) - 2z_{0}|p_{n} - p|  \\
 & > h_{p}(y_{p})-  h_{p}( y^{*}_{p}) - \frac{2z_{0}}{n}\\
 & > \frac{ h_{p}(y_{p})-  h_{p}( y^{*}_{p})}2 > 0,
\end{align*} for all $n$ sufficiently large, where the first inequality follows from Lemmas \ref{lem-p-hp-cont} and  \ref{lem-h-new}. 
\end{enumerate} Combining the three cases, we arrive at  \begin{align*}
|h_{p}& (y^{*}_{p_{n}})- h_{p}( y^{*}_{p})| \\ & \ge \rho: =\min\bigg\{h_{p}( y^{*}_{p})- h_{p}( y^{*}_{p}+ \e_{0}), h_{p}( y^{*}_{p}-  \e_{0}) -h_{p}( y^{*}_{p}),  \frac{ h_{p}(y_{p})-  h_{p}( y^{*}_{p})}2\bigg  \} > 0,
\end{align*} for all $n$ sufficiently large. On the other hand, we have 
\begin{align*} 
 |F^{*}_{p_{n}} -F^{*}_{p}| & = |h_{p_{n}}(y^{*}_{p_{n}})- h_{p}( y^{*}_{p})|      
   \ge |h_{p}(y^{*}_{p_{n}})- h_{p}( y^{*}_{p})| - |h_{p_{n}}(y^{*}_{p_{n}})- h_{p}( y^{*}_{p_{n}})| 
   > \rho - \frac{z_{0}} n,
\end{align*} where the last inequality follows from Lemma \ref{lem-p-hp-cont} as in the previous case. Again, this leads to a contradiction 
 thanks to Lemma \ref{lem-p-Fp*cont}. 
\end{proof}

 We are now ready to present the main result of this section, which proves the equilibrium result in \thmref{thm-main-MFGC}. Recall the supply rate function $\mathfrak z$ defined in \eqref{eq-kappa-Q} and price function $\vphi$ given in Condition \ref{c-cond}(b) as well as the vector-valued function $\Psi$ defined in \eqref{e:Psi(p)-fn}.   Proposition \ref{prop-Rwy} shows that the long-term average supply rate of the $(w, y)$-policy is equal to ${\mathfrak z}(w,y)$. Furthermore, Proposition \ref{prop-z=Bid/Bxi}(iii) implies that ${\mathfrak z}(w,y) \in (0, z_{0}]$ for any $(w, y)\in \cR$. 
\begin{thm} \label{prop-game-equi}
Suppose Conditions \ref{diff-cnd}, \ref{c-cond},  
 and  \ref{3.9-suff-cnd} hold.     Then  
\begin{itemize}
  \item[{\em(i)}]   the mapping \begin{equation}
\label{e-MFG-fxpt}
 \mathfrak z\circ\Psi\circ \vphi: [0, z_{0}]\mapsto [0,z_{0}]  
\end{equation} has a fixed point $z^{e} \in [0, z_{0}]$; 
  \item[{\em(ii)}] denoting  $p^{e} : = \vphi ( z ^{e})$ and $(w^{e}, y^{e})= \Psi\circ\vphi(z^{e})$, the $(w^{e}, y^{e})$-policy is an admissible  mean-field equilibrium strategy for problem \eqref{e:sec2game}.
 \end{itemize} 
\end{thm}
\begin{proof} (i) We have observed in Remark \ref{rem-about-condition2.6} that Condition  \ref{c-cond}(b) implies that 
 $\vphi(z) \in \La$ for every $z\in [0, z_{0}]$. Consequently the function $\Psi\circ \vphi: [0, z_{0}]\mapsto\cR$ is continuous thanks to Condition~\ref{c-cond}(b) and Proposition \ref{prop-wp-yp-cont}. Obviously $\mathfrak z$ is a continuous function on $\cR$. Furthermore, in view of Proposition \ref{prop-z=Bid/Bxi}(iii), $\mathfrak z(w, y) \in [0, z_{0}]$ for all $(w, y) \in \cR$. Therefore, $ \mathfrak z\circ\Psi\circ \vphi$ is a continuous function from $  [0, z_{0}]$ to $  [0, z_{0}]$. Hence we conclude from the Brouwer fixed point theorem that the function $\mathfrak z\circ\Psi\circ \vphi$ has a fixed point $z^{e} \in [0, z_{0}]$.

 (ii) Let  $p^{e}$ and  $(w^{e}, y^{e})$ be as in statement of the theorem. 
For the $(w^{e}, y^{e})$-policy $R^{(w^{e}, y^{e})}$  
 given in Definition \ref{thresholds-policy},  
 we have from  \eqref{eq-kappa-Q} and \eqref{e:F_K}  that \begin{equation}
\label{eq:MFG-value}
\kappa^{R^{(w^{e}, y^{e})}} = \mathfrak z(w^{e}, y^{e}) = z^{e}, \quad \text{ and }\quad J(R^{(w^{e}, y^{e})}, z^{e}) =  F^{*}_{p^{e}} = F_{p^{e}}(w^{e}, y^{e}).
\end{equation}   
Recall that $R^{(w^{e}, y^{e})} \in  \A$.  On the other hand, for any $R\in   \A$,  Proposition \ref{prop-Fmax}    implies that $J(R, z^{e}) \le F^{*}_{p^{e}} $. Consequently, $R^{(w^{e}, y^{e})}$   is a mean field equilibrium strategy for problem \eqref{e:sec2game}  in the class  $ \A$. 
 \end{proof}

\begin{rem}\label{rem-equi-inAp}  
 The conclusion of Theorem  \ref{prop-game-equi} improves the results in \cite{Chris-21}, which establishes the existence of an equilibrium  {\em in the class of  threshold strategies when $a$ is an entrance boundary and $c\equiv 0$}.  Here we show that for the more general problem \eqref{e:reward-fn}--\eqref{e:sec2game}, the $(w^{e}, y^{e})$-policy $R^{(w^{e}, y^{e})}$ is an  equilibrium in   $\A$, the set of {\em all} admissible impulse control policies. In addition, in view of Remark \ref{rem1-transversality},   $R^{(w^{e}, y^{e})}$   is actually a mean field equilibrium strategy for problem \eqref{e:sec2game}  in a larger class, namely,  $\A_{p^{e}}$, where $p^e$ is defined in Theorem \ref{prop-game-equi} and $\A_p$ is defined in \remref{rem1-transversality}.
\end{rem}

\section{Mean Field Control}\label{sect-mean-field-control}
This section is devoted to the mean field control problem \eqref{e:sec2control}. In other words,   we wish to   find a   policy $Q^{\star} \in \A $ that maximizes the reward functional $J(R,\kappa^{R})$. 

As mentioned in the introduction, we leverage the mean field structure and renewal theory to explicitly solve \eqref{e:sec2control}. A crucial step is establishing the upper bound for $J(R,\kappa^{R})$ for all $R\in \A$. To achieve this, we first consider a family of {\em constrained} classical long-term average stochastic impulse control problems in \eqref{e-mean-field-control}.  Using a Lagrange multiplier,  this leads to an associated family of unconstrained problems on the right-hand side of \eqref{e-lagrange-bnd}.  Thanks to Proposition \ref{prop-Fmax},   the $(w^*(\lambda, z),  y^*(\lambda, z))$-policy is an optimal strategy for the unconstrained problem, where $(w^*(\lambda, z),  y^*(\lambda, z))$ is the maximizing pair for the function $F_{\vphi(z)-\la}$, with  $\la \in \Lambda_z$ and $\Lambda_z$ being defined in \eqref{e:Lambda-z-defn}.  An asymptotic analysis of  $  y^*(\lambda, z)$ as $\lambda  \uparrow  \sup \Lambda_z$ leads to the key fixed-point identity \eqref{e:z=Bid/Bxi}. This identity allows us to derive an upper bound for $J(R,\kappa^{R})$ in Theorem~\ref{thm-mean-field-control}, which is attained by a specific $(w,y)$-policy, thus proving its optimality for the mean-field control problem.

Recall that we have observed in \eqref{e:kappa<z0} that $\kappa^{R} \le z_{0}$ for any $R\in  \A$,   where $z_{0}$ is defined in \eqref{eq:z0defn} and is finite  under Condition \ref{diff-cnd}.   Therefore,  we now consider the following family of constrained long-term average impulse  control problems: 
  \begin{equation}
\label{e-mean-field-control}
\begin{cases}
\displaystyle  
\sup_{R\in \mathcal A} \liminf_{t\to\infty} [ \lan  c,\mu_{0,t}^{R}\ran + \lan  \vphi(z)  B\id- K,\mu_{1,t}^{R}\ran ], \\
 \text{ subject to } \ z= \limsup_{t\to\infty}   \lan   B\id,\mu_{1,t}^{R}\ran, 
\end{cases}
\end{equation}  where $z\in [0, z_{0}]$.    The constraint in \eqref{e-mean-field-control} is a direct consequence of the definition of $\kappa^{R}$ in \eqref{e:kappaQ}.  
To solve the constrained problem \eqref{e-mean-field-control}, we   consider the following unconstrained problem by the Lagrange multiplier method. That is,  for any given $z\in [0, z_{0}]$,  we consider  
\begin{equation}
\label{e-lagrange-reformulation}\begin{aligned}
  \sup_{R\in \mathcal A} & \bigg\{
   \liminf_{t\to\infty} [ \lan  c,\mu_{0,t}^{R}\ran + \lan  \vphi(z)  B\id- K,\mu_{1,t}^{R}\ran ]- \lambda \Big ( \limsup_{t\to\infty}   \lan   B\id,\mu_{1,t}^{R}\ran-z\Big)\bigg\}, \ \ \la \in \R. \end{aligned}
\end{equation}

\begin{lem}\label{lem5.1}
For any   $\lambda \in \R$,  $z\in [0, z_{0}]$, and $R \in \A$,  we have  \begin{align}
\label{e-lagrange-bnd}\nonumber  
& \liminf_{t\to\infty} [ \lan  c,\mu_{0,t}^{R}\ran + \lan  \vphi(z)  B\id- K,\mu_{1,t}^{R}\ran ]- \lambda\Big ( \limsup_{t\to\infty}   \lan   B\id,\mu_{1,t}^{R}\ran-z\Big)\\ & \quad \le   
 \limsup_{t\to\infty} [ \lan   c,\mu_{0,t}^{R}\ran + \lan ( \vphi(z) -\lambda)  B\id- K,\mu_{1,t}^{R}\ran ] + \lambda  z.
\end{align} 
\end{lem} 
\begin{proof}
  Arbitrarily fix an $R \in \A$ and let $\lambda, z $ be as in the statement of the lemma.   Let $\{t_{j}\}$ be a sequence satisfying   $\lim_{j\to \infty} t_{j} =\infty$ and 
   $$ \limsup_{t\to\infty}   \lan   B\id,\mu_{1,t}^{R}\ran = \lim_{j\to\infty}  \lan   B\id,\mu_{1,t_{j}}^{R}\ran .$$ We next choose a subsequence $\{t_{j_{k}}\}$  of $\{t_{j}\}$ satisfying $$ \limsup_{j\to\infty} [ \lan   c,\mu_{0,t_{j}}^{R}\ran + \lan  \vphi(z)  B\id- K,\mu_{1,t_{j}}^{R}\ran ]  = \lim_{k\to\infty }  [ \lan   c,\mu_{0,t_{j_{k}}}^{R}\ran + \lan  \vphi(z)  B\id- K,\mu_{1,t_{j_{k}}}^{R}\ran ].  $$   Then  
    we have 
\begin{align*} 
 &  \liminf_{t\to\infty} [ \lan c,\mu_{0,t}^{R}\ran + \lan  \vphi(z)  B\id- K,\mu_{1,t}^{R}\ran ]- \lambda\Big ( \limsup_{t\to\infty}   \lan   B\id,\mu_{1,t}^{R}\ran-z\Big) \\ 
 &\ \  \le  \limsup_{j\to\infty}  [ \lan   c,\mu_{0,t_{j}}^{R}\ran + \lan  \vphi(z)  B\id- K,\mu_{1,t_{j}}^{R}\ran ]-\lambda  \lim_{j\to\infty}  \lan   B\id,\mu_{1,t_{j}}^{R}\ran + \la z\\ 
 & \ \   = \lim_{k\to \infty}  [ \lan c,\mu_{0,t_{j_{k}}}^{R}\ran + \lan  ( \vphi(z) -\la)  B\id- K,\mu_{1,t_{j_{k}}}^{R}\ran ] + \la z \\
 & \ \ \le   \limsup_{t\to\infty} [ \lan   c,\mu_{0,t}^{R}\ran + \lan ( \vphi(z) -\lambda)  B\id- K,\mu_{1,t}^{R}\ran ] + \lambda  z,  
\end{align*}  establishing \eqref{e-lagrange-bnd}. 
\end{proof}

In view of the right-hand side of  \eqref{e-lagrange-bnd} and Proposition \ref{prop-Fmax}, 
 we shall now  consider the function  
$ F_{\vphi(z) -\lambda}(w, y)$, $(w, y)\in \cR$ for  $ z \in [0,z_0]$ and $\lambda \in (-\infty,\vphi(z))$, 
 where  $F_{p}$ is defined in \eqref{e:F_K} using $p = \vphi(z) - \lambda \in \R$.

 Recall the set of feasible prices $\La$ defined in \eqref{e:set-Lambda-defn} as well the price function $\vphi$ satisfying Condition \ref{c-cond}(b). 
For every fixed $z \in [0, z_{0}]$, define \begin{align}\label{e:Lambda-z-defn}\Lambda_{z}: =  \vphi(z) -\La  = \{\vphi(z) - p: p\in \La\}.
\end{align} 
 Note that for every  $\lambda \in \Lambda_{z}$, we have 
   $\vphi(z) - \lambda \in \La$.  In addition, we have observed in Section \ref{sec-equilibrium} that $p_{0} = \inf\La \ge \frac{K}{b-a}$ and that $0 \notin \La$. Hence $\vphi(z) - \lambda >0$ for every  $\lambda \in \Lambda_{z}$. Furthermore,   in view of  Lemma \ref{lem-La set}, if Conditions  \ref{diff-cnd} and \ref{c-cond} hold, then  for every $z \in [0, z_{0}]$, we have $\Lambda_{z} \neq \emptyset$ with $\la_{z}^{\mathrm r} : = \sup  \Lambda_{z} = \vphi(z) - p_{0} \le \vphi(z) $; and $\la_{1} \in \Lambda_{z}$ whenever $\la_{1} < \la_2 $ and $\la_{2} \in \Lambda_{z}$. 
Consequently, we can write $\Lambda_{z} = (-\infty, \la_{z}^{\mathrm r})$ or $\Lambda_{z} = (-\infty, \la_{z}^{\mathrm r}]$.

\begin{lem}\label{sect5:propF-max} Assume Conditions \ref{diff-cnd},  \ref{c-cond},  and \ref{3.9-suff-cnd} hold.  Let $z \in [0,z_0]$.  
\begin{itemize}
  \item[{\em(i)}] For any $\la\in \Lambda_{z}$, there exists a unique pair $(w^{*}, y^{*})  = (w^{*}(\lambda, z), y^{*}(\lambda, z)) \in \cR$ so that 
\begin{equation}\label{sect5:e-Fmax}
 h_{\vphi(z) -\lambda}(w^{*}) \ge \sup_{(w, y) \in \cR}   F_{\vphi(z) -\lambda}(w, y) = F_{\vphi(z) -\lambda} (w^{*}, y^{*}) = h_{\vphi(z) -\lambda}(y^{*}). 
\end{equation}

  \item[{\em(ii)}] For any  $\la_{1} < \la_{2} \in \Lambda_{z}$, we have $$\sup_{(w, y) \in \cR}F _{\vphi(z) -\lambda_{2}} (w, y)  <  \sup_{(w, y) \in \cR}F _{\vphi(z) -\lambda_{1}} (w, y) .$$
\item[{\em(iii)}] Furthermore, $\lim_{\la\uparrow \la_{z}^{\mathrm r}} \sup_{(w, y) \in \cR}F _{\vphi(z) -\lambda} (w, y)  = \bar c(b).$
\end{itemize} \end{lem}

\begin{proof} Assertion (i) follows from Proposition~\ref{prop-Fmax} directly  since $\la \in \Lambda_z$ implies that $\vphi(z)-\la$ satisfies   \cndref{cond-interior-max}  and hence \eqref{sect5:e-Fmax} holds for some $(w^{*}, y^{*})  = (w^{*}(\lambda, z), y^{*}(\lambda, z)) \in \cR$. Assertion (ii) is obvious as $\vphi(z) -\lambda_{1} > \vphi(z) -\lambda_{2}$. 

We now prove (iii). For every $\la \in \Lambda_{z}$, we have $ \vphi(z) -\lambda \in \La$ and hence 
$F^{*}_{ \vphi(z) -\lambda} : = \sup_{(w, y)\in \cR} F_{ \vphi(z) -\lambda}(w, y) > \bar c(b)$. Furthermore, in view of 
assertion (ii), $F^{*}_{ \vphi(z) -\lambda}  $  decreases  as $\la \uparrow \la_{z}^{\mathrm r}$. Thus the limit  $\lim_{\la\uparrow \la_{z}^{\mathrm r}} F^{*}_{ \vphi(z) -\lambda} $ exists and is greater or equal to $  \bar c(b)$. 

Suppose   that $\lim_{\la\uparrow \la_{z}^{\mathrm r}} F^{*}_{ \vphi(z) -\lambda} = \bar c(b) + \delta$ for some  $\delta> 0$. Then there exists an $\e > 0$ so that $F^{*}_{\vphi(z) -\la_{z}^{ \mathrm r} + \rho} \ge  \bar c(b) + \frac34\delta $ for all $\rho \in (0, \e]$. 
 Furthermore, since $\la_{z}^{\mathrm r} -\rho \in \Lambda_{z}$,   there exists a pair $(w^{*}_{\rho}, y^{*}_{\rho}) \in \cR$ so that  \begin{align*} 
\bar c(b) + \frac34\delta  & \le F^{*}_{\vphi(z) - \la_{z}^{\mathrm r} +\rho} =F_{\vphi(z) -  \la_{z}^{\mathrm r} +\rho } (w^{*}_{\rho}, y^{*}_{\rho})\\ 
& =  \frac{Bg(w^{*}_{\rho}, y^{*}_{\rho}) + (\vphi(z) -\la_{z}^{\mathrm r} - \rho)B\id (w^{*}_{\rho}, y^{*}_{\rho}) -K}{ B \xi(w^{*}_{\rho}, y^{*}_{\rho})} + 2\rho  \frac{B\id (w^{*}_{\rho}, y^{*}_{\rho})}{B \xi(w^{*}_{\rho}, y^{*}_{\rho})}\\ 
& = F_{ \vphi(z)- \la_{z}^{\mathrm r} -\rho } (w^{*}_{\rho}, y^{*}_{\rho}) +  2\rho \frac{1}{\xi'(\theta)} \\
& \le F_{ \vphi(z)- \la_{z}^{\mathrm r} -\rho } (w^{*}_{\rho}, y^{*}_{\rho}) +  2\rho z_{0}, \end{align*} where the last equality follows from the mean value theorem,   $\theta$ is between $w^{*}_{\rho}$ and $ y^{*}_{\rho}$, and $z_{0}$ is the positive constant defined in \eqref{eq:z0defn}.
 We now let $\rho_{0} : = \frac\e2 \wedge \frac\delta {8z_{0}}$. Then  it follows from the above   displayed equation  that \begin{displaymath}
\bar c(b) + \frac34\delta  
\le  F_{\vphi(z)- \la_{z}^{\mathrm r} -\rho_{0} } (w^{*}_{\rho_{0}}, y^{*}_{\rho_{0}})  +  2  \frac\delta {8z_{0}} z_{0} = F_{\vphi(z)- \la_{z}^{\mathrm r} -\rho_{0} } (w^{*}_{\rho_{0}}, y^{*}_{\rho_{0}})  + \frac\delta 4,
\end{displaymath} or \begin{displaymath}
F_{\vphi(z)- \la_{z}^{\mathrm r} -\rho_{0} } (w^{*}_{\rho_{0}}, y^{*}_{\rho_{0}})  \ge \bar c(b) + \frac \delta 2.
\end{displaymath} This says that $\la_{z}^{\mathrm r} +\rho_{0} \in \Lambda_{z}  $, where $\Lambda_{z}$ is defined in \eqref{e:Lambda-z-defn}.  
   Hence $\la_{z}^{\mathrm r} <  \la_{z}^{\mathrm r} +\rho_{0}  \le \sup\Lambda_{z} $, contradicting the fact that $\la_{z}^{\mathrm r} = \sup \Lambda_{z}$. 
\end{proof}

\begin{lem}\label{prop-ybar-la-to-b} Assume Conditions \ref{diff-cnd},   \ref{c-cond}(a), and \ref{3.9-suff-cnd} hold. 
Then for any $z\in [0, z_{0}]$ and $\la \in \Lambda_{z}$,  the maximizing pair $( w^{*}(\la,  z),   y^{*}(\la, z))$ for the function $F_{\vphi(z) -\lambda}(w, y)$ satisfies \linebreak
$\lim_{\la\uparrow \la_{z}^{\mathrm r}}  y^{*}(\la, z)  =b$.
\end{lem}

\begin{proof} Fix an arbitrary   $z\in [0, z_{0}]$,
let $\{\la_{k}\}$ be an increasing sequence that converges  to $\la_{z}^{\mathrm r}  $ as $k\to\infty$, and denote by $\{(w^{*}_{\la_{k}}, y^{*}_{\la_{k}}) = (w^{*}({\la_{k}},z), y^{*} ({\la_{k}},z))\}$ a corresponding sequence of maximizers for the functions $\{F_{\vphi(z) -\lambda_{k}}\}$; here and throughout the proof we omit the dependence on $z$ in the sequence $\{(w^{*}_{\la_{k}}, y^{*}_{\la_{k}})\}$ for notational simplicity. 
  If $\liminf_{k\to\infty} y^{*}_{\la_{k}} < b$, then there exists a subsequence with $\lim_{j \to \infty}  y^{*}_{\la_{k_{j}}} =:  y^{*} < b$ and \begin{align}\label{e:5.7}
\nonumber \bar c(b) & \le \lim_{j\to\infty} F_{ p_{k_{j}}}(w^{*}_{\la_{k_{j}}}, y^{*}_{\la_{k_{j}}}) =  \lim_{j\to\infty}    h_{ p_{k_{j}}}(y^{*}_{ \lakj}) \\ &=  \lim_{j\to \infty} \frac{g'(y^{*}_{\lakj}) + \vphi(z) - \lakj}{\xi'(y^{*}_{\lakj}) } = \frac{g'(y^{*}) + \vphi(z) - \la_{z}^{\mathrm r}}{\xi'(y^{*}) } = h_{ p_{0}}(y^{*}).
  \end{align}  where $p_{k_{j}} = \vphi(z) - \la_{k_{j}} \in \La$. 
  We have  $\lim_{j\to\infty} p_{k_{j}} =p_{0} $ thanks to the definition of $\{\la_{k}\}$   and the fact that  $\la_{z}^{\mathrm r} = \sup\Lambda_{z} = \vphi(z) -p_{0}$, where $p_{0} = \inf\La$.
  
    The rest of the proof is divided into two cases. 

{\it Case 1: $ \la_{z}^{\mathrm r} = \vphi(z)$}.  We claim that $h_{\vphi(z) -\la_{z}^{\mathrm r}}(y^{*}) < \bar c(b)$ if $ \la_{z}^{\mathrm r} = \vphi(z)$, thus yielding a contradiction to \eqref{e:5.7}. Indeed,  if $ \la_{z}^{\mathrm r} = \vphi(z)$, then $p_{0}=0$. In view of \cndref{c-cond}(a) and \eqref{e:h'-expression}, the function $ h_{p_{0}}=  h_{0} $ is strictly increasing on a neighborhood of $b$. Hence $$h_{0}(y^{*} ) < \lim_{y\to b} h_{0}( y) = \bar c(b), $$ where the last equality follows from  Lemma \ref{lem-h-new}. 
Thus we must have \begin{displaymath}
b \le \liminf_{k\to\infty} y^{*}_{\la_{k}}  \le \limsup_{k\to\infty} y^{*}_{\la_{k}}  \le b.
\end{displaymath} This gives $\lim_{\la\uparrow \la_{z}^{\mathrm r}}  y^{*}_{\la_{k}}  =b$.

{\it Case 2: $\la_{z}^{\mathrm r} < \vphi(z)$}. Let  $\{(w^{*}_{\la_{k}}, y^{*}_{\la_{k}})\}$ and $y^{*}$ be as before.   Since the sequence  $\{ w^{*}_{\la_{k_{j}}}\}$ is bounded, there exists  a further subsequence, still denoted by $\{ w^{*}_{\la_{k_{j}}} \}$  
     with a slight abuse of notation, and some $w^{*} \le y^{*}$ so that $\lim_{j \to \infty}  w^{*}_{\la_{k_{j}}} =:  w^{*} $.
Thanks to \eqref{sect5:e-Fmax}, we have 
\begin{equation}\label{e:F*la_kj}
h_{p_{k_{j}}} ( w^{*}_{\la_{k_{j}}}) \ge F^{*}_{p_{k_{j}}} =F_{p_{k_{j}}} (w^{*}_{\lakj}, y^{*}_{\lakj}) =  h_{p_{k_{j}} } ( y^{*}_{\la_{k_{j}}}).  
\end{equation} 
Due to the monotonicity of $c$ and $\mu$ on $(a, \hat x_{\mu,c})$ and \eqref{e:h'-expression}, it follows that the function $h_{p_{k_{j}}} $ is strictly increasing on $(a, \hat x_{\mu,c})$. Together with \eqref{e:F*la_kj}, this implies that  $h_{ p_{k_{j}}} $ has a local maximum at some $\bar x_{\lakj}  \in  (w^{*}_{\lakj}, y^{*}_{\lakj}) $ with $h'_{ p_{k_{j}}} (\bar x_{\lakj}) = 0$.  Since the sequence  $\{ \bar x_{\la_{k_{j}}}\}$ is bounded, there exists  a further subsequence, still denoted by $\{ \bar x_{\la_{k_{j}}} \}$, and some $\bar x \le y^{*}$ so that $\lim_{j \to \infty}  \bar x_{\la_{k_{j}}} =:  \bar x $. Note that $w^{*} \le \bar x\le y^{*}$.
     
      We have \begin{displaymath}
\lim_{j \to \infty} r_{p_{k_{j}}} (\bar x_{\lakj})  = \lim_{j\to\infty} [c(\bar x_{\lakj}) + (\vphi(z) - \lakj)\mu(\bar x_{\lakj})] = c(\bar x) +(\vphi(z) - \la_{z}^{\mathrm r}) \mu(\bar x) =: r_{p_{0}}(\bar x),
\end{displaymath} and 
     \begin{equation}\label{e:ell-limit la_r}
\lim_{j \to \infty} h_{p_{k_{j}}} (\bar x_{\lakj}) = \lim_{j\to \infty} \frac{g'(\bar x_{\lakj}) + \vphi(z) - \lakj}{\xi'(\bar x_{\lakj}) } = \frac{g'(\bar x) + \vphi(z) - \la_{z}^{\mathrm r}}{\xi'(\bar x) } = h_{p_{0}} (\bar x),
\end{equation} Using these equations  in \eqref{e:h'-expression}  yields 
\begin{align*}
0& =\lim_{j \to \infty} h'_{p_{k_{j}}} (\bar x_{\lakj})  = \lim_{j\to\infty} \frac{m(\bar x_{\lakj})}{M[a, \bar x_{\lakj}]} [r_{p_{k_{j}}} (\bar x_{\lakj}) - h_{p_{k_{j}}}(\bar x_{\lakj})] 
\\& = \frac{m(\bar x)}{M[a, \bar x]} [r_{p_{0}}(\bar x) -h_{p_{0}}(\bar x) ] = h'_{p_{0}} (\bar x).
\end{align*} 

 The fact that $h'_{p_{0}}(\bar x) = 0$ allows us to define $ y_{p_{0}} : = \min\{x\in \I:  h'_{p_{0}}( x)= 0\}$. Note that  $\hat x_{\mu,c} <   y_{p_{0}}\le \bar x \le y^{*} < b$. 
Moreover, using the same arguments as those for Lemma \ref{lem-h-new}, we can show that   that  $h_{p_{0}}(\cdot)  $ is strictly increasing on $(a, y_{p_{0}})$ and strictly decreasing on $( y_{p_{0}}, b)$.
Since $y^{*} \in [y_{p_{0}}, b)$,    it follows from Lemma \ref{lem-h-new}  that 
\begin{equation}\label{e:h0(y*)>c(b)}
h_{p_{0}}(y^{*}) > \lim_{y\to b} h_{p_{0}}(y) = \bar c(b).
\end{equation} On the other hand, using the facts that $p_{k_{j}} \to p_{0}$ and  $y^{*}_{\lakj} \to y^{*}$ as $j\to\infty$, a similar calculation as that in  \eqref{e:ell-limit la_r}  yields   $\lim_{j\to\infty} h_{p_{k_{j}}}(y^{*}_{\lakj}) = h_{p_{0}}(y^{*}) $.  Furthermore, in view of \eqref{e:F*la_kj} and Lemma \ref{sect5:propF-max} (iii), we have \begin{displaymath}
h_{p_{0}}(y^{*})  = \lim_{j\to\infty} h_{p_{k_{j}}}(y^{*}_{\lakj})  =  \lim_{j\to\infty} F^{*}_{p_{k_{j}}} = \bar c(b);
\end{displaymath} contradicting \eqref{e:h0(y*)>c(b)}. Hence we must have  $\lim_{k\to\infty} y^{*}_{\la_{k}} = b$. The proof is complete. 
\end{proof}

\begin{prop}\label{lem-la-z} Under the conditions of Lemma \ref{prop-ybar-la-to-b},  
for any $z\in (0, z_{0})$,  
 there exists a $\la_{z} \in \Lambda_{z}$ such that \begin{equation}\label{e:z=Bid/Bxi}
 z = \mathfrak z\circ\Psi (\vphi(z) - \la_{z})= \frac{y^{*} (\la_{z}, z) - w^{*} (\la_{z}, z)}{\xi(y^{*} (\la_{z}, z)) - \xi(w^{*} (\la_{z}, z))},
\end{equation} where the functions $\Psi$  and $\mathfrak z$ are defined in \eqref{e:Psi(p)-fn}  and \eqref{eq-kappa-Q}, respectively, and $$(w^{*}(\lambda_{z}, z),  y^{*}(\lambda_{z}, z) ) =\Psi (\vphi(z) - \la_{z})= \argmax_{(w, y) \in \cR} F_{\vphi(z) - \la_{z}} (w, y).$$  
\end{prop}
\begin{proof}  Note that for any $z\in (0, z_{0})$ and $\la \in \Lambda_{z}$, $\vphi(z) - \la \in \La$. Therefore $\Psi (\vphi(z) - \la)$ and $ \mathfrak z\circ\Psi (\vphi(z) - \la)$ are well-defined. The rest of the proof is divided into several steps. 

   {\em Step 1.} Thanks to Lemmas 
   3.5 and 3.6 of \cite{HelmSZ:26} and Condition \ref{c-cond}(a), for any $z\in (0, z_{0})$ and $\la \in \Lambda_{z}$ with $\vphi(z) -\la > 0$,  we have 
\begin{displaymath}
\sup_{(w, y) \in \cR} \frac{  Bg(w, y) - K}{(\vphi(z) -\la) B\xi(w, y)}  < \infty,
\ \text{ and hence } \
\lim_{\la\to-\infty}\sup_{(w, y) \in \cR} \frac{  Bg(w, y) - K}{(\vphi(z) -\la) B\xi(w, y)} = 0.
\end{displaymath} Thus for any $\e > 0$, there exists a $\la_{\e}  \in (-\infty, \vphi(z))  $ so that 
$$\sup_{(w, y) \in \cR} \frac{  Bg(w, y) - K}{(\vphi(z) -\la) B\xi(w, y)}  < \e \ \text{ for all } \ \la \le \la_{\e}.$$ 
This, in turn, implies  that  for all $(w, y) \in \cR$ and $\la \le \la_{\e}$,
\begin{align*}
F_{\vphi(z) - \la}(w, y) & = (\vphi(z) -\la)\bigg(\frac{y-w}{\xi(y) - \xi(w)} +  \frac{  Bg(w, y) - K}{(\vphi(z) -\la) B\xi(w, y)}   \bigg) 
 <  (\vphi(z) -\la) (\mathfrak z(w, y) + \e), 
\end{align*} 
or 
\begin{displaymath}
\mathfrak z(w, y)   > \frac{F_{\vphi(z) - \la}(w, y)}{\vphi(z) -\la} - \e. 
\end{displaymath} 
This holds in particular for the maximizing pair $(w^{*}, y^{*})  = (w^{*}(\lambda, z), y^{*}(\lambda, z)) \in \cR$, whose existence  follows from Lemma \ref{sect5:propF-max}(i). On the other hand, for any $z \in  (0, z_{0})$, let $\delta $ be a positive number so that $ z + \delta <  z_{0}$.  Proposition  \ref{prop-z=Bid/Bxi}(ii) implies that there exists a pair $(\wdt w, \wdt y) \in \cR$ so that $\mathfrak z(\wdt w, \wdt y)=z+ \delta$.   Then we have 
\begin{align*} 
 \mathfrak z(w^{*}, y^{*}) 
 &  > \frac{F_{\vphi(z) - \la}(w^{*}, y^{*})}{\vphi(z) -\la} - \e 
  \ge  \frac{F_{\vphi(z) - \la}(\wdt w, \wdt y)}{\vphi(z) -\la} - \e  \\
 & = \mathfrak z(\wdt w, \wdt y)  
 + \frac{   Bg(\wdt w,\wdt  y) - K}{(\vphi(z) -\la) B\xi(\wdt w, \wdt y)} -\e\\
 & = z+\delta  + \frac{   Bg(\wdt w,\wdt  y) - K}{(\vphi(z) -\la) B\xi(\wdt w, \wdt y)} -\e.
\end{align*} 
Since $\lim_{\la\to - \infty}   \frac{   Bg(\wdt w,\wdt  y) - K}{(\vphi(z) -\la) B\xi(\wdt w, \wdt y)}  =0$,   there exists a $\wdt\la_{\e} < \la_{\e}$ so that $ \frac{   Bg(\wdt w,\wdt  y) - K}{(\vphi(z) -\la) B\xi(\wdt w, \wdt y)}  > -\e$ for all $\la <  \wdt\la_{\e}$. Plugging this observation into the above displayed equation, we have \begin{displaymath}
  \mathfrak z(w^{*}, y^{*}) >   z+\delta -2\e, \quad \forall \la <  \wdt\la_{\e}.
\end{displaymath} Since $\e > 0$ is arbitrary, we have $$\lim_{\la \to -\infty}  \mathfrak z\circ\Psi(\vphi(z) -\la) =    \lim_{\la \to -\infty}   \mathfrak z(w^{*}, y^{*})  \ge   z+\delta > z.$$   

{\em Step 2.} We now show that \begin{equation}\label{5e:Bid/Bxi=0}
\lim_{\la \uparrow \la_{z}^{\mathrm r}}  \mathfrak z\circ\Psi (\vphi(z) - \la) 
    = \frac{ y^{*}(\la, z)  -  w^{*}(\la, z) }{ \xi(y^{*}(\la, z))  - \xi( w^{*}(\la, z))} = 0. 
\end{equation} To this end, let $\{\la_{k}\}$ be an increasing sequence that converges to $\la_{z}^{\mathrm r}$. To simplify notation, let us denote $(w^{*}_{\la_{k}}, y^{*}_{\la_{k}}) : = (w^{*}(\la_{k},z), y^{*}(\la_{k}, z)) $.     Lemma  \ref{prop-ybar-la-to-b} implies that  $\lim_{k\to \infty} y^{*}_{\la_{k}}  = b$. We have either $\limsup _{k\to\infty} w^{*}_{\la_{k}}  <  b$ 
 or $\limsup _{k\to\infty} w^{*}_{\la_{k}}  = b$.   In the former case,   we can pick a $w_{0} \in \I$ so that $w^{*}_{\la_{k}} < w_{0} <  y^{*}_{\la_{k}}$ for all $k$ sufficiently large. Then for all such $k$'s, we can write \begin{equation}\label{e2:Bid/Bxi}
0 \le \frac{B\id( w^{*}_{\la_{k}}, y^{*}_{\la_{k}})}{B\xi( w^{*}_{\la_{k}}, y^{*}_{\la_{k}})} = \frac{y^{*}_{\la_{k}} - w^{*}_{\la_{k}}  }{\xi(y^{*}_{\la_{k}}) - \xi(w_{0}) + \xi(w_{0}) - \xi(w^{*}_{\la_{k}}) }   \le  \frac{y^{*}_{\la_{k}} - a}{\xi(y^{*}_{\la_{k}}) - \xi(w_{0})}; 
\end{equation} the right-most expression of \eqref{e2:Bid/Bxi} converges to 0 as $k\to\infty$ due to \eqref{eq:xi_limit_b} if $b < \infty$;  if $b =\infty$, it  converges to 0 as $k\to\infty$ thanks to L'H\^opital's rule and \eqref{e-sM-infty}.  
We now consider the case $\limsup _{k\to\infty} w^{*}_{\la_{k}}  = b$, in which $b < \infty$; the case when $b=\infty$ can be handled in a similar fashion. For any $\e > 0$, thanks to Lemma 3.4 of \cite{HelmSZ:26}, 
there exists some $0 < \delta <  b-a  $ so that \begin{equation}\label{e:s5-Bid/Bxi}
\mathfrak z(w, y) < \e, \text{ for all } b-\delta< w< y < b.
\end{equation} 
Since  $\lim_{k\to\infty}y^{*}_{\la_{k}} = b$,  there exists some $K_{1} \in \NN$ so that $y^{*}_{\la_{k}} > b-\delta$ for all $k\ge K_{1}$. For each $k\ge K_{1}$, if $w^{*}_{\la_{k}}>  b-\delta$, then \eqref{e:s5-Bid/Bxi} says that 
\begin{displaymath}
0 \le \mathfrak z( w^{*}_{\la_{k}}, y^{*}_{\la_{k}})  < \e. 
\end{displaymath} 
If $w^{*}_{\la_{k}}\le b-\delta$, using the same argument as that for \eqref{e2:Bid/Bxi}, we can pick some $K_{2} \ge K_{1}$ so that  
\begin{displaymath}
0 \le  \mathfrak z( w^{*}_{\la_{k}}, y^{*}_{\la_{k}}) = \frac{B\id( w^{*}_{\la_{k}}, y^{*}_{\la_{k}})}{B\xi( w^{*}_{\la_{k}}, y^{*}_{\la_{k}})} \le  \frac{y^{*}_{\la_{k}} - a}{\xi(y^{*}_{\la_{k}}) - \xi(b-\delta)} < \e, \quad \text{ for all } k \ge K_{2}. 
\end{displaymath} 
Combining these two  cases, we have $0 \le  \mathfrak z( w^{*}_{\la_{k}}, y^{*}_{\la_{k}})  < \e $ 
  for all $k \ge K_{2}$. Since $\e> 0$ is arbitrary, we have   $\lim_{k\to\infty}  \mathfrak z( w^{*}_{\la_{k}}, y^{*}_{\la_{k}})  =0$; 
    establishing \eqref{5e:Bid/Bxi=0}. 

As a consequence of \eqref{5e:Bid/Bxi=0},   there exists some $\la < \la_{z}^{\mathrm r}$ so that $ \mathfrak z\circ\Psi (\vphi(z) - \la)  < z$.

{\em Step 3.} Finally, thanks to Propositions \ref{prop-wp-yp-cont} and \ref{prop-z=Bid/Bxi}, the function from $\Lambda_z$ to $[0,z_0]$  
 \begin{displaymath}
\la \mapsto  \mathfrak z\circ\Psi (\vphi(z) - \la) 
\end{displaymath}
is continuous. This, combined with the conclusions of Steps 1 and 2, implies that there exists a $\la_{z} \in \Lambda_{z}$ for which \eqref{e:z=Bid/Bxi} holds. The proof is complete.  
\end{proof}

We are now ready to present the main result of this section, which establishes the existence of an optimal MFC control result in \thmref{thm-main-MFGC}.

\begin{thm}\label{thm-mean-field-control} 
Assume Conditions \ref{diff-cnd},  \ref{c-cond},  
and    \ref{3.9-suff-cnd} hold. Then 
\begin{itemize}
  \item[{\em(i)}] for any $R\in  \A$,  we have $J(R, \kappa^{R}) \le  \sup_{(w, y) \in \cR}\Upsilon(w, y),  $ where \begin{displaymath}
  \Upsilon(w, y): =   \frac{ g(y) - g(w) + \vphi(\mathfrak z(w, y))(y-w) -K}{\xi(y) - \xi(w)}, \quad (w, y) \in \cR;
\end{displaymath}  
  \item[{\em(ii)}] there exists some $( w^{\star},   y^{\star}) \in \cR$ such that $\sup_{(w, y) \in \cR}\Upsilon(w, y) =\Upsilon(w^{\star},   y^{\star}) $;
  \item[{\em(iii)}]  the $(w^{\star},   y^{\star})$-policy $Q^{\star}\in   \A$ satisfies $J(Q^{\star}, \kappa^{Q^{\star}}) = \Upsilon(w^{\star},   y^{\star}) $. Thus $Q^{\star}$ is an admissible  optimal policy for the mean field control problem \eqref{e:sec2control}. 
\end{itemize}
 \end{thm}
\begin{proof} (i) Note that $\kappa^R =0$ for any $R\in   \AF$. On the other hand, for any $R\in  \AI$,  we  can    use Proposition \ref{prop-z=Bid/Bxi}(iv) and \eqref{e2-transversality}  to derive $$\kappa^{R}= \limsup_{t\to\infty} \lan B\id, \mu^{R}_{1,t}\ran = \limsup_{t\to\infty} \lan \frac{B\id}{B\xi} B\xi, \mu^{R}_{1,t}\ran < z_{0}. $$ Combining these two cases, we have $\kappa^{R} < z_{0}$ for all $R\in  \A$.

Let $z\in (0, z_{0}) $ and consider an arbitrary $R\in  \A$   with $z = \kappa^{R}$. We use the value $\la_{z} $  and the corresponding optimizing pair $( \hat{w},  \hat{y})= ( \hat{w}(\la_{z}, z),  \hat{y}(\la_{z}, z))$ from Proposition \ref{lem-la-z}; here we use the notation $\,\hat{}\,$ rather than $^{*}$ for the pair $( \hat{w},  \hat{y})$  in order to avoid confusion with the pair $( w^{\star},   y^{\star})$ appearing in the statement of the theorem.  Since $\la_{z}\in \Lambda_{z}$, we have $\vphi(z) - \la_{z} \in \La$. 
 Then it follows that  
\begin{align} \label{e1-sect5-thm-proof}
\nonumber J(R, \kappa^{R}) & =  \liminf_{t\to\infty} [ \lan  c,\mu_{0,t}^{R}\ran + \lan  \vphi(z)  B\id- K,\mu_{1,t}^{R}\ran ] \\
 \nonumber&   =  \liminf_{t\to\infty} [ \lan   c,\mu_{0,t}^{R}\ran + \lan  \vphi(z)  B\id- K,\mu_{1,t}^{R}\ran ]- \lambda_{z}\Big ( \limsup_{t\to\infty}   \lan   B\id,\mu_{1,t}^{R}\ran-z\Big) \\ 
 \nonumber&   \le  \limsup_{t\to\infty} [ \lan   c,\mu_{0,t}^{R}\ran + \lan ( \vphi(z) -\lambda_{z})  B\id- K,\mu_{1,t}^{R}\ran ] + \lambda_{z}  z \\ 
\nonumber &   \le  \sup_{(w, y) \in \cR} F_{\vphi(z) - \la_{z}}(w, y) + \la_{z} z \\
\nonumber &   =   \frac{  g( \hat{y})- g( \hat{w}) + (\vphi(z) - \la_{z}) ( \hat{y} -  \hat{w}) - K}{\xi( \hat{y})- \xi( \hat{w})} + \la_{z} z\\ 
 \nonumber &   = \frac{  g( \hat{y})- g( \hat{w}) + \vphi(\mathfrak z( \hat{w}, \hat{y} ) )  ( \hat{y} -  \hat{w}) - K}{\xi( \hat{y})- \xi( \hat{w})}\\ 
 & = \Upsilon ( \hat{w}, \hat{y}), 
\end{align} where the first and second inequalities  follows from \eqref{e-lagrange-bnd} and Proposition \ref{prop-Fmax}, respectively, and the last equality follows from   Proposition \ref{lem-la-z}.

 Now we consider an arbitrary $R\in  \A$ with $\kappa^{R} =  \limsup_{t\to\infty}   \lan   B\id,\mu_{1,t}^{R}\ran= 0$. Recall the set $\Lambda_{0}$ given in \eqref{e:Lambda-z-defn}. Also let  $\la^{\mathrm r}_{0}: =\sup  \Lambda_{0}$ and   $\e > 0$ be an arbitrary positive number.  Note that $\vphi(0)-\la_{0}^{\mathrm r}= p_{0} = \inf \La$. Moreover, thanks to Lemma \ref{sect5:propF-max}(iii), for any $\e> 0$,  we can pick a $\delta > 0$ so that $F^{*}_{\vphi(0)-\la_{0}^{\mathrm r}+\delta} < \bar c(b) + \e$. 
  Then we have 
\begin{align*} 
 J(R, \kappa^{R}) & =  \liminf_{t\to\infty} [ \lan  c,\mu_{0,t}^{R}\ran + \lan  \vphi(0)  B\id- K,\mu_{1,t}^{R}\ran ] \\
 &   =  \liminf_{t\to\infty} [ \lan   c,\mu_{0,t}^{R}\ran + \lan  \vphi(0)  B\id- K,\mu_{1,t}^{R}\ran ]- (\la_{0}^{\mathrm r} -\delta)\Big ( \limsup_{t\to\infty}   \lan   B\id,\mu_{1,t}^{R}\ran-0\Big) \\ 
 &  \le  \limsup_{t\to\infty} [ \lan   c,\mu_{0,t}^{R}\ran + \lan ( \vphi(0) -\la_{0}^{\mathrm r}+\delta)  B\id- K,\mu_{1,t}^{R}\ran ]\\
 &   \le F^{*}_{ \vphi(0) -\la_{0}^{\mathrm r}+\delta} \\
 & < \bar c(b)+ \e,
 \end{align*} where we used \eqref{e-lagrange-bnd} and Proposition \ref{prop-Fmax} (b) to derive the first two inequalities.  
  Since $\e> 0$ is arbitrary, we have \begin{equation}
\label{e2-sect5-thm-proof}
 J(R, \kappa^{R}) \le 
  \bar c(b).  
\end{equation}

A combination of \eqref{e1-sect5-thm-proof} and \eqref{e2-sect5-thm-proof} gives us 
\begin{equation}\label{e:sec5-J(R,R)-bd}
J(R, \kappa^{R}) \le \sup_{(w, y) \in \cR}  \Upsilon (w,y) \vee \bar c(b). 
\end{equation}

(ii) We now show that there exists a pair $( w^{\star},   y^{\star}) \in \cR$ such that $\sup_{(w, y) \in \cR}\Upsilon(w, y) =\Upsilon(w^{\star},   y^{\star}) $ and that $\Upsilon(w^{\star},   y^{\star})  > \bar c(b)$. To this end, we note that for any $(w, y) \in \cR$, we have $\mathfrak z(w, y) 
\in (0, z_{0})$ and hence \begin{displaymath}
  \vphi_{\min} \le \vphi\big( \mathfrak z(w, y)\big) \le \vphi_{\max} < \infty,
\end{displaymath} where  $\vphi_{\min}: = \min_{z\in [0, z_{0}]} \vphi(z)$ and  $\vphi_{\max}: = \max_{z\in [0, z_{0}]} \vphi(z)$. Consequently, \begin{align}\label{e:Upsilon<F_max} \nonumber 
F_{\vphi_{\min}}(w, y) & = \frac{ g(y) - g(w) + \vphi_{\min}(y-w) -K}{\xi(y) - \xi(w)} \le \Upsilon (w, y)\\ & \le  \frac{ g(y) - g(w) + \vphi_{\max}(y-w) -K}{\xi(y) - \xi(w)}= F_{\vphi_{\max}}(w, y).
\end{align} At one hand, the assumption  that $\vphi_{\min} \in \La$ implies that there exists a pair $(\wdt w, \wdt y)\in \cR$ so that \begin{displaymath}
\Upsilon (\wdt w,\wdt y)\ge  F_{\vphi_{\min}} (\wdt w,\wdt y) > \bar c(b).
\end{displaymath} On the other hand, the proof of Proposition \ref{prop-Fmax} reveals that the maximum value of  $F_{\vphi_{\max}}$ (and hence $\Upsilon$) on the boundary of $\cR$ is less than or equal to $\bar c(b)$. Therefore the maximum value of $\Upsilon$ is achieved at some point $( w^{\star},   y^{\star}) \in \cR$ with \begin{displaymath}
\Upsilon(w^{\star},   y^{\star}) =\sup_{(w, y) \in \cR}\Upsilon(w, y) > \bar c(b).
\end{displaymath} Note that the maximizing pair $( w^{\star},   y^{\star})$ may have $w^{\star} =a$ if $a$ is an entrance boundary. Assertion (ii) is established. This, in turn, leads to assertion (i) thanks to \eqref{e:sec5-J(R,R)-bd}.

(iii) Finally we notice that the $ ( w^{\star},   y^{\star})$ policy  $Q^{\star}$ 
is in $ \A$ with $\kappa^{Q^{\star}} = \mathfrak z(w^{\star}, y^{\star})   = :z$  thanks to  \eqref{eq-kappa-Q}. Moreover, using \eqref{e:F_K},  the long-term average reward of $Q^{\star}$ is equal to $$J(Q^{\star}, \kappa^{Q^{\star}}) =  F^{\star}_{\vphi(z)} =\frac{g(y^{\star})- g(w^{\star}) + \vphi(z)  (y^{\star} - w^{\star}) - K}{\xi(y^{\star})- \xi(w^{\star})}= \Upsilon(w^{\star},   y^{\star}).$$   The proof is complete. 
\end{proof}

\begin{rem}\label{rem-MFC-cond26}
  {   We discuss the role of Condition \ref{c-cond}(b). We observed in Remark \ref{rem-about-condition2.6} that it is equivalent to requiring $\vphi_{\min} \in \La$. Suppose it fails and that $\vphi_{\max} \notin \La$. Then $F^*_{\vphi_{\max}} =  \sup_{(w, y) \in \cR} F_{\vphi_{\max}} (w, y)  \le \bar c(b) $. Consequently, for any $R\in \A$, we have from \eqref{e:sec5-J(R,R)-bd} and \eqref{e:Upsilon<F_max} that \begin{align*}
    J(R, \kappa^R) \le \sup_{(w, y) \in \cR}  \Upsilon (w,y) \vee \bar c(b) \le \sup_{(w, y) \in \cR} F_{\vphi_{\max}} (w, y) \vee \bar c(b)  = F^*_{\vphi_{\max}}  \vee \bar c(b)\le  \bar c(b).
  \end{align*}  Recall from \eqref{e:reward-zero-policy} that $J(\mathfrak R, \kappa^{\mathfrak R}) = \bar c(b)$. Thus, in this case, the do-nothing  policy $\mathfrak R$ is an optimal policy for the MFC problem \eqref{e:sec2control}.}
\end{rem}

\begin{rem}\label{rem-optim-inAp}  
Theorem \ref{thm-mean-field-control} asserts that the $(w^{\star},   y^{\star})$-policy $Q^{\star}$ is an optimal mean field impulse   strategy in  the class $ \A$. 
In contrast, \cite{Chris-21} only derives the optimality of a threshold impulse strategy in the class of {\em stationary strategies}. 
\end{rem}

\begin{rem}[Comparison of MFG and MFC] \label{rem-MFG-MFC} 
Theorems \ref{prop-game-equi} and \ref{thm-mean-field-control} show that   equilibrium and   optimal MFC strategies exist and both are of threshold type policies under Conditions \ref{diff-cnd}, \ref{c-cond},  
and \ref{3.9-suff-cnd}. Moreover, thanks to \eqref{eq:MFG-value}, the MFG value is given by  $F_{p^e}(w^e,y^e) = \Upsilon(w^e,y^e)$, where $(w^e,y^e) \in \cR$ and $p^e= \vphi(\mathfrak z(w^e,y^e))$ are determined in the statement of Theorem \ref{prop-game-equi}. Compare this with Theorem \ref{thm-mean-field-control} and it is obvious that the equilibrium MFG value is less than or equal to  the optimal value of the  MFC, which is equal to $\sup_{(w, y)\in \cR} \Upsilon(w, y)$. This difference stems from MFC's centralized maximization of collective reward, compared to MFG's focus on individual agent strategies. On the other hand, MFG is more robust to individual deviations, as agents cannot improve their reward  by deviating from the equilibrium.  MFC, however, relies on centralized enforcement for optimality and is less robust to such deviations.  Indeed, given the optimal MFC supply rate $z^\star$ and  the corresponding unit price $\vphi(z^\star)$, an individual agent, if permitted, might  deviate from the optimal MFC strategy by selecting an alternative policy and thereby attain a higher individual reward. An implication of this is that the optimal MFC control  is not necessarily an equilibrium for  the MFG.

To illustrate these differences, we  study two examples in the next section. 

 \end{rem}

\section{Examples}\label{sect:example}
\begin{example}\label{exam-logistic}
We consider a stochastic logistic  growth model given by the SDE:  
\begin{equation}\label{eq-logistic}
dX_{0}(t) = r X_{0}(t)\bigg(1 -  \frac{X_{0}(t)}{\delta}\bigg) dt + \sigma X_{0}(t) dW(t), \quad X_{0}(0) =1,
\end{equation}  where $W$ is a one-dimensional standard Brownian motion, and  $r, \delta$, and $  \sigma$ are positive constants. It is straightforward to verify that  the state space of $X_{0}$ is  $(0, \infty)$, with both 0 and $\infty$ being natural boundaries. In addition, the scale function $S$ and the speed measure $M$ are absolutely continuous with respect to the Lebesgue measure with densities 
 $$s(x) = 
 x^{-\alpha} e^{\theta (x-1)},  \quad  
 \text{ and }\quad m(x) = \frac{2}{\sigma^{2}x^{2} s(x)} =  
 \frac{2}{\sigma^{2}} x^{\alpha - 2} e^{- \theta (x-1)}.  $$ where $\alpha: =\frac{2r}{\sigma^2} $ and $\theta: =\frac{2r}{\delta\sigma^2}  $.
 We have\begin{displaymath}
S(0, y] = \int_{0}^{y}  x^{-\alpha} e^{\theta (x-1)}dx, 
\end{displaymath} and $$M(0, y] = \int_{0}^{y} m(x) dx =   \frac{2e^{\theta} }{\sigma^{2}}   \int_{0}^{y} x^{\alpha- 2} e^{-\theta x} dx =   \frac{2 \theta^{1-\alpha}e^{\theta} }{\sigma^{2}}     \gamma\left( \alpha-  1, \theta y\right),
$$ where $\gamma$ is the lower incomplete gamma function $
\gamma(s, z) = \int_0^z t^{s-1} e^{-t} dt.
$ 
Then 
\begin{displaymath}
\xi(x) =\int_{1}^{x} M[0, v] dS(v) = \frac{2\theta^{1-\alpha}}{\sigma^{2}}   \int_1^x \gamma\left( \alpha-  1, \theta v\right) v^{-\alpha} e^{\theta v} dv.
\end{displaymath}

 If $2r > \sigma^{2}$ or $\alpha > 1$, detailed computations reveal that $$S(0, y] = \infty, \quad M(0, y]  <\infty,  \ \text{ for any } y > 0,  \quad \text{ and } \lim_{x\to\infty} s(x) M(0, x] =\infty. $$  This verifies Condition \ref{diff-cnd}(a,b); Condition \ref{diff-cnd}(c) trivially holds. In addition,   we have \begin{displaymath}
M(0, \infty) =   \frac{2 \theta^{1-\alpha}e^{\theta} }{\sigma^{2}}     \Gamma( \alpha-  1 ) < \infty, 
\end{displaymath} where $\Gamma(\alpha) : = \int_{0}^{\infty} x^{\alpha-1} e^{-x} dx$ is the gamma function, and  $\lim_{x\downarrow 0} \xi(x) =  -\infty.$ Next we compute 
\begin{displaymath}
\ell(x) = 
\frac1{\xi'(x)} =\frac1{ s(x) M(0, x]} = \frac{\int_{0}^{x} \mu(u) dM(u)}{M(0, x]} =  \frac{\sigma^{2}   x^{ \alpha} }{2\theta^{1-\alpha}    e^{\theta x}    \gamma\big(\alpha  - 1, \theta x\big)}.
\end{displaymath} In view of Lemma \ref{lem-h-new},  the maximum value  $
z_{0} := \sup_{x> 0} \frac1{\xi'(x)} $  occurs at $x^{*}$, 
  where $x^{*}$ is the unique solution to the equation $\ell'(x)  =0$, which leads to \begin{equation}\label{eq-x*-exm}
e^{\theta x} \gamma(\alpha-1, \theta x) (\alpha-\theta x) = \theta^{\alpha-1} x^{\alpha-1}.
 \end{equation} Note that the function $\mu(x) =r x(1-\frac{x}{\delta})$ is strictly increasing on $(0, \hat x_{\mu,c})$ and strictly decreasing on $(\hat x_{\mu,c}, \infty)$, where $\hat x_{\mu,c} = \frac{\delta}{2 }$. Thus it follows that $x^{*} > \hat x_{\mu,c}$. On the other hand, \eqref{eq-x*-exm} implies that $ \alpha-\theta x^{*} > 0$ or $x^{*} <  \delta$. Using \eqref{eq-x*-exm}, we can rewrite \begin{displaymath}
z_{0} = \ell(x^{*}) =\frac{\sigma^{2}}{2}x^{*}(\alpha- \theta x^{*}). 
\end{displaymath} Since $x^{*} \in (\frac \delta{2 }, \delta)$, we have $ 0  < z_{0} < \frac{\sigma^{2}}{2} \frac \delta{2 }\big(\alpha- \theta \frac \delta2\big) = \frac{\delta r}{4}.$
 
 Now we consider,  for illustrative purposes,  the price function $\varphi(z) := \frac{3 }{  3 + z+2\sin z},  z \ge 0$, as well as the running reward function $c(x) : =1-e^{-x}, x\ge 0 $. Obviously, both Conditions \ref{c-cond} and \ref{3.9-suff-cnd} are satisfied. Note that $\vphi$ is not a monotone function. Moreover, we have 
 \begin{displaymath}
\bar c(b) = \int_{0}^{\infty} c(x) \pi(dx) = \frac1{M(0, \infty)} \int_{0}^{\infty} c(x) m(x) dx = 1- \bigg( \frac{2r}{2r +\delta\sigma^{2}}\bigg)^{  \alpha-1}.
\end{displaymath} The function $g(x): = \int_{1}^{x} \int_{0}^{v} c(y) dM(y) dS(v), x> 0$ doesn't have an analytic form.  

For numerical demonstration, we set $r=\delta= 5,  \sigma=1$,  and $K=0.5$.  Numerical calculations reveal that $$\vphi_{\min}: = \min_{z\in [0, z_0]} \vphi(z) \ge \min_{z\in [0, \frac{\delta r}{4}]} \vphi(z) = 0.326668 \in \La. $$  This verifies  Condition \ref{c-cond}(b). Consequently, by Theorems \ref{prop-game-equi} and \ref{thm-mean-field-control}, mean field game equilibrium and optimal  mean field control strategies exist and admit  explicit characterizations.   The numerical results are summarized in    Table \ref{table1}. Note that the optimal value of the mean field control problem exceeds the equilibrium value of the mean field game problem by 0.242790.
\begin{table}[ht]
  \centering   \caption{Numerical Results of  Mean Field Game and Control Problems for Example \ref{exam-logistic}}
  \begin{tabular}{c|ccccc}
\hline
 Problem & $w^e$ or $w^{\star}$ & $y^e$ or $y^{\star}$ & Supply Rate & Price & Value\\ \hline
 MFG   &  1.279499  & 5.368681 &  5.221743 & 0.463276 & 2.674072 \\
 MFC  &  1.106232 &  6.306876 &  4.559874 & 0.537337 & 2.916862 \\ \hline
\end{tabular}
\label{table1}
\end{table} 

However, the optimal mean field control policy is not robust in the sense that an individual agent may achieve a superior long-term average reward if the unit price of impulses is set to  $p =\vphi(z^{\star})$, in which $z^{\star}$ is the optimal supply rate of the optimal mean field control policy. Indeed,  in this numerical example, with  the optimal  mean field control supply rate $z^{\star} = 4.559874$ and thus the price $p = \vphi(z^{\star}) = 0.537337$, an individual agent can adopt a different  $(w, y)$-policy with $w=1.326678$ and $ y= 5.216696$ and achieve a  long-term average reward of 3.064301, which is 0.147439 greater than the optimal mean field control value.
\end{example}

\begin{example}\label{exm2-LOksendal}
  We consider a population growth model in a stochastic environment proposed by \cite{L-Oksendal}:
  \begin{equation}\label{eq-pop2}
    dX(t)= rX(t)({ b}-X(t))dt + \sigma X(t) ({ b}- X(t))d W(t),
  \end{equation} where $W$ is a standard one-dimensional Brownian motion, $r>0$ is the growth rate, $b > 0$ is the carrying capacity, and $\sigma > 0$ is the volatility. The state space of $X$ is $\I = (0, b)$, with both 0 and $b$ being natural boundaries. The scale function $S$ and the speed measure $M$ are absolutely continuous with respect to the Lebesgue measure with densities
  \begin{displaymath}
    s(x) = x_0^\beta (b-x_0)^{-\beta} (b-x)^\beta x^{-\beta}, \quad m(x) = \frac{2}{\sigma^2} x_0^{ -\beta} (b-x_0)^{\beta} x^{\beta-2} (b-x)^{-\beta -2}, \ \ x\in (0,b),
  \end{displaymath} where  $x_{0}\in \I$  and $\beta : = \frac{2r}{b \sigma^{2}} >1$. One can verify that $0$ is nonattracting, with 
  \begin{align*}
    S(0, y]  = \infty, \   M[0, y] < \infty, \ \forall y\in (0, b), \quad \text{ and } \quad 
    \lim_{x\to b} s(x) M[0, x] = \infty.
  \end{align*} This verifies Condition \ref{diff-cnd}.  Moreover, detailed calculations reveal that $ \lim_{y\to b} M[0,y] =\infty$. 
  
  We next take the price and the running reward functions to be \begin{displaymath}
\vphi(z) : = \frac{2}{   3+ z + \cos(2  z)}, \ z\ge 0, \text{ and } c(x) : = 1- e^{-3x} + 0.01 x^{0.25},\  x \in (0, b).
\end{displaymath} It is obvious that both Conditions \ref{c-cond}(a) and \ref{3.9-suff-cnd} are satisfied.  

For numerical demonstration, we set $r=0.75, b= 5, \sigma=0.5$, and $K=0.2$. As in the previous example, we can verify Condition  \ref{c-cond}(b) numerically, which, in turn, establishes the existence of mean-field equilibrium and optimal mean-field control policies.   The numerical results are summarized in Table \ref{table2}. Note that the optimal value of the mean field control problem exceeds the equilibrium value of the mean field game problem by 0.282482.

\begin{table}[ht]
  \centering   \caption{Numerical Results of  Mean Field Game and Control Problems for Example \ref{exm2-LOksendal}}
  \begin{tabular}{c|ccccc}
\hline
 Problem & $w^e$ or $w^{\star}$ & $y^e$ or $y^{\star}$ & Supply Rate & Price & Value\\ \hline
 MFG   &  2.707186&  4.889822 &  2.560956 &  0.335620  & 1.249932 \\
 MFC  &  2.750384 & 4.997066  & 1.638624   & 0.548274    & 1.532414 \\ \hline
\end{tabular}\label{table2}
\end{table}
As we observed in the previous example, the optimal mean field control policy is not robust.  Corresponding to the price $p =\vphi(z^{\star})= 0.548274$, an individual agent may adopt the    $(2.787973, 4.737556)$-policy and achieve a  long-term average reward of  1.834061, which is 0.301648  greater than the optimal mean field control value.
\end{example}

%

\appendix
  \section{Appendix}  
\label{sect-appendix}
  This appendix presents several technical proofs.  

  \begin{proof} [Proof of Proposition~\ref{prop-z=Bid/Bxi}] 
(i) We first observe  that thanks to \eqref{e-sM-infty}, 
  $\lim_{x\to b} \frac{1}{\xi'(x)} =0$. On the other hand, since $a$ is a non-attracting point by Condition \ref{diff-cnd}, we have $S(a, y] =\infty$ for any $y \in \R$ and hence $s(a+) =\infty$. Then we can apply   L'H\^ospital's rule to compute   
  \begin{align*}
    \lim_{x\downarrow a} \frac{1}{\xi'(x)} =   \lim_{x\downarrow a} \frac{\frac{1}{s(x)}}{M(a, x] } = \lim_{x\downarrow a} \frac{\mu(x) \frac{2}{s(x) \sigma^{2}(x)}}{m(x)} =\mu(a),
  \end{align*}  which is finite thanks to Condition \ref{diff-cnd}(c).  Therefore $z_{0} < \infty$ as desired.

(ii)  Note that $\frac1{\xi'(x)} = \frac1{s(x) M[a, x]} > 0$ for any $x\in \I$. This together with the observation that  $\lim_{x\to b} \frac{1}{\xi'(x)} =0$ 
 imply that  there exists a   $ y_{0}   \in \I$ so that $\sup_{x\in \I} \frac{1}{\xi'(x)} = \frac{1}{\xi'(y_{0})}  $. In case there are multiple  maximizers for the function $\frac{1}{\xi'(x)}$, we pick $y_{0} $ to be the largest one.  
 
 Let $z\in (0, z_{0})$. 
 Using the continuity of the function $\frac{1}{\xi'(x)}$, we can find a pair $(w, y_{1}) $ so that  
 $ y_{0} < w < y_{1} < b$  and $z < \frac{1}{\xi'(x)}$  
 for all $x\in [w, y_{1}]$. 
  We now  consider the continuous function \begin{equation}
\label{e-sec2-f-defn}
 f(x) = f(x; w, z) : = z(\xi(x) - \xi(w)) - (x-w), \ \ x\ge w.
\end{equation} We have $f(w) = 0$. Moreover, by the mean value theorem, \begin{equation}
\label{e1-prop32-pf}
f(y_{1}) =z(\xi(y_{1}) -\xi(w)) - (y_{1}-w) = (y_{1}-w)  ( z \xi'(\theta) - 1)< 0,
\end{equation} where $\theta \in (w, y_{1})$.  Since $b$ is natural, in view of Table 7.1 on p.~250 of \cite{KarlinT81}, 
  we have \begin{equation}
\label{eq:xi_limit_b}
\lim_{x\to b} (\xi(x) - \xi(y_{0}) ) = \int_{y_{0}}^{b} M[a, v] dS(v) \ge  \int_{y_{0}}^{b} M[y_{0}, v] dS(v) =\infty   \quad \forall y_{0}\in \I. 
\end{equation} 
This implies that  \begin{equation}
\label{e:x/xi(x)-b-limit}
\lim_{x\to b}  \frac{x-w}{\xi(x) - \xi(w)} = 0
\end{equation} if $b < \infty$. For the case when $b=\infty$, the limit in \eqref{e:x/xi(x)-b-limit} is still true thanks to L'H\^opital's Rule and \eqref{e-sM-infty}. Therefore,  it follows from  \eqref{eq:xi_limit_b} and  \eqref{e:x/xi(x)-b-limit}
 that \begin{align} \label{e2-prop32-pf}
\lim_{x\to b} f(x) & = \lim_{x\to b} (\xi(x) - \xi(w)) \left[z-  \frac{x-w}{\xi(x) - \xi(w)} \right]= \infty. 
\end{align}
     In view of \eqref{e1-prop32-pf} and \eqref{e2-prop32-pf}, we conclude that there exists a $y > y_{1}$ so that $f(y) = 0$ or $\frac{y-w}{\xi(y) - \xi(w)} = z$. 
      
      (iii) The assertion that $\frac{y-w}{\xi(y) - \xi(w)} \le z_{0}$ for any  $  (w,  y) \in \cR$ is obvious thanks to the mean value theorem and the definition of $z_{0}$. 
      
      (iv) We now prove assertion (iv). Put $\ell : = \frac{1}{\xi'}$.  Suppose there exists a pair $  (w,  y) \in \cR$ with $\frac{y-w}{\xi(y) - \xi(w)}= z_{0}$, then there exists a $\theta \in (w, y)$ so that $\ell (\theta) = \frac1{\xi'(\theta)} = z_{0}$, 
        In other words, $\theta$ is an extreme point for the function $\ell$ and hence $\ell'(\theta) = 0$, where \begin{equation}
\label{e:ell'-expression}
\ell'(x) = \frac{m(x)}{M[a, x]} (\mu(x) - \ell(x)) = \frac{m(x)}{M^{2}[a, x]}  \int_{a}^{x}(\mu(x) -\mu(y)) dM(y) =0.
\end{equation} Since $\mu$ is strictly increasing  on $(a, \hat x_{\mu,c})$, we have $\theta > \hat x_{\mu,c}$ and there exists  a subinterval $[x_{1}, x_{2}] \subset [\hat x_{\mu,c},\theta]$ on which $\mu$ is strictly decreasing. Without loss of generality, we can pick $x_{1}$ so that $$x_{1}: = \min\{x\ge \hat x_{\mu,c}: \mu \text{ is strictly decreasing on }[x, y] \text{ for some } y > x\}.$$ In particular, $x_{1} < \theta< y$.  Since $\mu$ is also concave on $[x_{1}, b) \subset (\hat x_{\mu,c}, b)$, it follows that $\mu$ is strictly decreasing on $[x_{1}, b)$. 
      
      On the other hand, for the function $f(\cdot) = f(\cdot; w, z_{0})$ defined in \eqref{e-sec2-f-defn}, we have $f(w) =f(y) =0$ and $f'(x) = z_{0} \xi'(x) -1 \ge 0$ for all $x\ge w$.  Therefore $f$ is constant and thus $f'(x) =0$ on $[w, y]$. This further implies that $\ell(x) = \frac1{\xi'(x)} = z_{0}$ and hence $\ell'(x) =0$ for all $x\in (w, y)$. In view of the first expression of $\ell'$ given in \eqref{e:ell'-expression}, we have $\mu(x) = \ell(x) = z_{0}$ for all $x\in (w, y)$. This is a contradiction to the conclusion that $\mu$ is strictly decreasing on $[x_{1}, b)$ since $(w\vee x_{1}, y) \subset (w, y) \cap [x_{1}, b)$. The proof is complete. 
  \end{proof} 

\begin{proof}[Proof of \eqref{eq:reward_equiv}]
  For any $R\in \A$ and $Q\in \A_0$, let $\{t_j\}$ be such that $\lim_{j\to\infty} t_j =\infty$ and \begin{align*}
  \hat J(R, Q) : = \lim_{j\to\infty}[\lan c, \mu_{0,t_j}^R\ran + \lan \vphi(\kappa_{t_j}^Q) B\id - K, \mu^R_{1, t_j}\ran].
\end{align*} Take a further subsequence $\{t_{j_k}\}$ of $\{t_j\}$ so that $\lim_{k\to\infty} t_{j_k} =\infty$ and \begin{align*}
  \liminf_{j\to\infty} [\lan c, \mu_{0,t_j}^R\ran + \lan \vphi(\kappa^Q) B\id - K, \mu^R_{1, t_j}\ran] = \lim_{k\to\infty} [\lan c, \mu_{0,t_{j_k}}^R\ran + \lan \vphi(\kappa^Q) B\id - K, \mu^R_{1, t_{j_k}}\ran]
\end{align*} Then  we have \begin{align*}
   J& (R, \kappa^Q)- \hat J(R, Q)\\ & \le \liminf_{j\to\infty} [\lan c, \mu_{0,t_j}^R\ran + \lan \vphi(\kappa^Q) B\id - K, \mu^R_{1, t_j}\ran] - \lim_{j\to\infty}[\lan c, \mu_{0,t_j}^R\ran + \lan \vphi(\kappa_{t_j}^Q) B\id - K, \mu^R_{1, t_j}\ran] \\ 
  & = \lim_{k\to\infty} [(\lan c, \mu_{0,t_{j_k}}^R\ran + \lan \vphi(\kappa^Q) B\id - K, \mu^R_{1, t_{j_k}}\ran) - (\lan c, \mu_{0,t_{j_k}}^R\ran + \lan \vphi(\kappa^Q_{t_{j_k}}) B\id - K, \mu^R_{1, t_{j_k}}\ran )] \\
   & =\lim_{k\to\infty}  [\vphi(\kappa^Q) - \vphi(\kappa_{t_{j_k}}^Q)]\lan  B\id, \mu^R_{1, t_{j_k}}\ran \\ 
  & \le \limsup_{k\to\infty} |\vphi(\kappa^Q) - \vphi(\kappa_{t_{j_k}}^Q)|\cdot \limsup_{k\to\infty} \lan B\id, \mu^R_{1, t_{j_k}}\ran\\ &  \le 0\cdot z_0 =0, 
\end{align*}  where the last inequality follows from the continuity of $\vphi$ and Lemma \ref{lem2-transversality}. 

On the other hand, we can take a sequence $\{t_i\}$ such that $\lim_{i\to\infty} t_i =\infty$ and \begin{align*}
  J(R, \kappa^Q) = \lim_{i\to\infty} [\lan c, \mu_{0,t_i}^R\ran + \lan \vphi(\kappa^Q) B\id - K, \mu^R_{1, t_i}\ran].
\end{align*} Let $\{t_{i_l}\}$ be a further subsequence of $\{t_l\}$ with $\lim_{l\to\infty} t_{i_l} =\infty$ and \begin{align*}
  \liminf_{i\to\infty} [\lan c, \mu_{0,t_i}^R\ran + \lan \vphi(\kappa_{t_i}^Q) B\id - K, \mu^R_{1, t_i}\ran] = \lim_{l\to\infty} [\lan c, \mu_{0,t_{i_l}}^R\ran + \lan \vphi(\kappa_{t_{i_l}}^Q) B\id - K, \mu^R_{1, t_{i_l}}\ran].
\end{align*} Also note that $0\le \liminf_{t\to\infty} \lan B\id, \mu_{1, t}^R\ran \le \limsup_{t\to\infty} \lan B\id, \mu_{1, t}^R\ran \le z_0$.  Then \begin{align*}
   J& (R, \kappa^Q)- \hat J(R, Q)\\ 
  & \ge \lim_{i\to\infty} (\lan c, \mu_{0,t_i}^R\ran + \lan \vphi(\kappa^Q) B\id - K, \mu^R_{1, t_i}\ran) -  \liminf_{i\to\infty} [\lan c, \mu_{0,t_i}^R\ran + \lan \vphi(\kappa_{t_i}^Q) B\id - K, \mu^R_{1, t_i}\ran] \\ 
  & = \lim_{l\to\infty} [(\lan c, \mu_{0,t_{i_l}}^R\ran + \lan \vphi(\kappa^Q) B\id - K, \mu^R_{1, t_{i_l}}\ran)- (\lan c, \mu_{0,t_{i_l}}^R\ran + \lan \vphi(\kappa_{t_{i_l}}^Q) B\id - K, \mu^R_{1, t_{i_l}}\ran)]\\ 
  & = \lim_{l\to\infty} [(\vphi(\kappa^Q) - \vphi(\kappa_{t_{i_l}}^Q))\lan  B\id, \mu^R_{1, t_{i_l}}\ran ]\\ 
  & \ge -\limsup_{l\to\infty} |\vphi(\kappa^Q) - \vphi(\kappa_{t_{i_l}}^Q)|\cdot \limsup_{l\to\infty} \lan B\id, \mu^R_{1, t_{i_l}}\ran \ge -0\cdot z_0 =0.
\end{align*}  Thus, we have $J(R, \kappa^Q) = \hat J(R, Q)$ for any $R\in \A$ and $Q\in \A_0$ as claimed in \eqref{eq:reward_equiv}.\end{proof}
  
\begin{proof}[Proof of Lemma \ref{lem-h-new}] Assertion (i) follows from Lemma 4.3 of \cite{HelmSZ:26}. To see assertion (ii), we first note that Conditions \ref{c-cond}(a) and  \ref{3.9-suff-cnd} imply that the function $r_p$ is strictly increasing on $(a,\wdh x_{\mu, c} )$; this together with the expression for $h'_{p}$ in \eqref{e:h'-expression} implies that $h_{p}$ is strictly increasing on $(a,  \wdh x_{\mu, c} )$. On the other hand, Proposition 4.2 of  \cite{HelmSZ:26} indicates that under Conditions \ref{diff-cnd}, \ref{c-cond}(a), and \ref{cond-interior-max}, the function $F_p$ of \eqref{e:F_K} has an interior maximizer $(w^*_p,y_p^*)$ for which   \eqref{e-1st-order-condition} or \eqref{e2-1st-order-condition} holds. This, together with Assertion (i), implies that implies that there exists an $ x > \wdh x_{\mu,c}$ so that $h_{p}'(x) =0$. Furthermore, we can define $y_{p} : = \min\{x \in \I: h_{p}'(x) =0\}$ as in \eqref{e-y-hat-p-defn}; note that $y_{p}> \wdh x_{\mu,c}$. Using  \eqref{e:h'-expression} again, there must exist a subinterval $[ \wdt x_p, \wdt w_p] \subset (\hat x_{\mu,c},  y_{p}) $ on which the function $r_{p}$ is strictly decreasing.  If there are multiple such intervals, we select  the one with the smallest $ \wdt x_{p}$.
 Furthermore, we can pick $ \wdh x_{p} \in [\hat x_{\mu,c},  \wdt x_{p}]$ so that  $r_{p}$ is strictly increasing on $(a, \wdh x_{p})$, constant on $[\wdh x_{p}, \wdt x_{p}]$,  and decreasing on $[ \wdt x_{p}, w_{p}]$. The concavity of  $r_{p}$ on  $( \hat x_{\mu, c}, b)$ actually implies that it is decreasing  on $(\wdh x_{p}, b)$ and strictly decreasing on  $(\wdt x_{p}, b)$. Then using the same calculations as those in the proof of Lemma 4.7 of \cite{HelmSZ:26}, we can show that $h'_p(x) < 0 $ for all $x> \wdh y_p$ and hence  $h_{p}$ is strictly increasing on $(a, y_p)$ and  decreasing on $(y_{p}, b)$. The proof is complete. 
 \end{proof} 

 \def\cprime{$'$} \def\polhk#1{\setbox0=\hbox{#1}{\ooalign{\hidewidth \lower1.5ex\hbox{`}\hidewidth\crcr\unhbox0}}}

\end{document}